\documentclass{article}

\usepackage[margin=3.4cm]{geometry}

\usepackage[utf8]{inputenc}

\usepackage[dvipsnames]{xcolor}

\usepackage{amsmath,amssymb,amsthm,environ}
\usepackage{mathtools}
\usepackage{yhmath}
\usepackage{calc}
\usepackage{tikz}
\usepackage{tikz-cd}
\usetikzlibrary{graphs,arrows,decorations.pathmorphing,decorations.markings,fit,positioning,hobby,arrows.meta}
\usepackage[scr]{rsfso}
\usepackage{longtable}

\usepackage{extpfeil}
\newcommand\scalemath[2]{\scalebox{#1}{\mbox{\ensuremath{\displaystyle #2}}}}

\tikzset{
    marrow/.style={decoration={markings,mark=at position 0.5 with {\arrow{#1}}}, postaction=decorate}
}
\tikzset{%
vcenter/.style = {baseline = {([yshift = -.5ex](0, 1.25))}},%
functor fill/.style = {fill = black!30!white, opacity = .6},%
functor draw/.style = {line width = 2pt, draw = black!50!white, opacity = .6}
}

\usepackage{xargs}

\usepackage[hidelinks]{hyperref} 
\usepackage[capitalize, nameinlink, noabbrev]{cleveref}

\usepackage{enumitem}
\usepackage{rotating}

\makeatletter
\newcommand\varitem[1]{\item[\textbf{A\arabic{enumi}\rlap{$#1$}.}]%
	\edef\@currentlabel{A\arabic{enumi}{$#1$}}}
\makeatother

\usepackage{multicol}
\usepackage{multirow}

\usepackage{soul}

\usepackage{subcaption}

\usepackage[
  style=numeric,
  sorting=nyt,
  maxbibnames=999
]{biblatex}
\theoremstyle{plain}
\newtheorem{theorem}{Theorem}[section]

\newtheorem{proposition}[theorem]{Proposition}

\newtheorem{corollary}[theorem]{Corollary}
\newtheorem{lemma}[theorem]{Lemma}

\theoremstyle{definition}
\newtheorem{definition}[theorem]{Definition}

\theoremstyle{remark}
\newtheorem{remark}[theorem]{Remark}
\newtheoremstyle{stepstyle}%
{}{}
{}{}
{}{.}
{ }
{\thmnumber{#2}.\ \bfseries\boldmath\thmnote{#3}}

\theoremstyle{stepstyle}

\DeclareMathOperator{\coker}{coker}
\DeclareMathOperator{\colim}{colim}
\DeclareMathOperator{\image}{im}

\def\Fun{\textnormal{Fun}}
\def\tame{\textnormal{Tame}}
\def\id{\textnormal{id}}
\def\Nat{\textnormal{Nat}}
\def\End{\textnormal{End}}

\def\vect{\textnormal{vect}_{\mathbb{F}}}
\def\rad{\textnormal{rad}}

\newcommand{\pos}{\mathcal{Q}}

\newcommand{\R}{\mathscr{R}}
\newcommand{\C}{\mathcal{C}}
\newcommand{\D}{\mathcal{D}}
\newcommand{\A}{\mathcal{A}}
\newcommand{\B}{\mathcal{B}}

\newcommand{\ch}{\mathbf{Ch}}

\usepackage{adjustbox}

\newcommand{\define}[1]{{\bf \boldmath{#1}}}

\title{Structure theorem for cofibrant tame functors indexed by posets of dimension $1$}

\author{Wojciech Chach{\'o}lski\footnote{KTH Royal Institute, Stockholm, Sweden, \texttt{wojtek@kth.se}} \and Barbara Giunti\footnote{Graz Institute of Technology, Austria, and SUNY Albany, USA, \texttt{bgiunti@albany.edu}} \and Claudia Landi\footnote{University of Modena and Reggio Emilia, Italy, \texttt{clandi@unimore.it}} \and Francesca Tombari \footnote{KTH Royal Institute, Stockholm, Sweden, and Max Plank Institute, Leipzig, Germany, \texttt{tombari@kth.se}  - \textbf{corresponding author}}}

\date{}

\begin{document}

\maketitle

\begin{abstract}
In this paper, we discuss certain circumstances
in which the category of tame functors inherits an abelian category structure with minimal resolutions and a model category structure with minimal cofibrant replacements. We present a structure theorem for cofibrant objects in the category of
tame functors indexed by realizations of posets of dimension $1$ with values in the 
category of chain complexes in an abelian category
whose all objects are projectives. 
Moreover, we introduce a general technique to generate indecomposable objects in the abelian category of functors indexed by finite posets.
\end{abstract}



\section{Introduction}

Representing objects by functors indexed by posets plays an important role across many areas of mathematics. 
Finding a structure theorem for some meaningful collections of such functors is crucial.
Our goal in this article is to present a structure theorem for cofibrant objects in the model category $\tame(\R(\pos), \ch(\C))$ of tame functors indexed by the realization\footnote{The realization of a poset is a construction that associates with a poset a ``continuous'' poset, mimicking the relation between the poset $\mathbb{R}$ and its subposet $\mathbb{Z}$. 
This construction was introduced in~\cite{realisations_posets} and reviewed here in  Section~\ref{sec:realizations}. 
In particular, the reader is warned that the realization of a poset is not the geometric realization of the nerve of the poset.} $\R(\pos)$ of a finite poset $\pos$  of dimension at most 1 
with values in the category of non-negative chain complexes in an abelian category $\C$ whose objects are projective (Theorem~\ref{thm_indecomposables}).
 
Such a structure theorem can then be used to assign simplifying invariants to any object in this category, in particular to any functor $\R(\pos)\to\C$ (thought of as a chain complex concentrated in degree $0$). 
Examples of such invariants have played a central role in practical settings like topological data analysis, where the structure theorem is key for doing computations 
(see, e.g., \cite{Zomorodian2005} in the case when $\C$ is the category of vector spaces, or \cite{abcw} in the case when $\C$ is the category of chain complexes over vector spaces, with cofibrant objects appearing naturally as filtered objects).

Our story, describing all the elements involved in our structure theorem, starts with minimality.
Minimality allows us to associate with any object in the category the invariants given by the indecomposables of a minimal cofibrant replacement.
This is why we focus on model structures: they give us the right language to apply our structure theorem to study general objects. 
A model structure becomes significantly more valuable when we can describe its (co)fibrant objects. In this case, we prove that this is indeed possible for the model category in question. For example, we show that minimal cofibrant replacements exist and that their indecomposables can be explicitly enumerated.

We then continue with tameness.
Functors indexed by finite posets have proven to be convenient to encode
homotopical and homological properties of objects for example
in algebra~\cite{Nakajima}, geometry~\cite{Reineke}, and data analysis~\cite{oudot2015}. 
This is because if $\D$ is a finite poset, then the category of functors $\Fun(\D, \C)$ inherits from $\C$ the abelian structure, in which minimal resolutions exist and can be constructed explicitly using the structure of $\D$, or the model category structure, in which minimal cofibrant replacements exist and can also be constructed explicitly using the structure of $\D$. 
Often, the finite posets relevant for a given situation share certain common features which can be expressed by these posets being subposets of another, typically infinite or continuous, poset $\pos$.
By taking the left Kan extensions along these subposet inclusions, we can embed $\Fun(\D, \C)$ into the category $\Fun(\pos, \C)$. 
This ambient category $\Fun(\pos, \C)$ provides a convenient place in which functors indexed by different finite posets can be compared and studied together.  
However, for analyzing homological and homotopical properties, the category $\Fun(\pos, \C)$ is too complex. 
For example, when $\pos$ is infinite, although the category $\Fun(\pos, \C)$ inherits an abelian structure from $\C$, some of its objects may fail to admit minimal resolutions. 
Moreover, if $\C$ is a model category, then even in the cases when we know how to endow  $\Fun(\pos, \C)$ with a related model structure, cofibrant replacements are not explicit as they are typically constructed using transfinite induction, and again minimal such replacements may not exist. 
Therefore, we focus on the full subcategory of  $\Fun(\pos, \C)$ given by the functors in which we are the most interested: 
left Kan extensions along finite subposet inclusions $\D\subset \pos$.
Such functors have previously been called tame, as for example in \cite{realisations_posets,scolamiero2017multidimensional}, and their category is denoted here by $\tame(\pos,\C)$.
In this paper, whenever we call a functor tame, we refer to this notion of tameness.

We therefore focus on identifying conditions on the poset $\pos$ that would guarantee that the category $\tame(\pos,\C)$ inherits an abelian structure from $\C$. 
One of the main obstacles is to ensure $\tame(\pos,\C)$ is closed under finite limits. 
One of our key results (Corollary~\ref{cor_closedcolim}) states that this obstacle vanishes if $\pos$ and $\C$ are such that every tame functor $\pos\to\C$ can be obtained as the left Kan extension along a finite subposet inclusion admitting a transfer (see Definition~\ref{def:transfer}).
Intuitively, 
a transfer in our setting is a finite encoding \cite{miller2020homological}
that collapses the infinite poset $Q$ onto a finite one ``reversing''
the finite poset inclusion.
The key observation (Corollary~\ref{dsgdfh}) is that the 
left Kan extension along a poset inclusion admitting a transfer is given by the precomposition with the transfer, which is an operation inherently preserving limits (Proposition~\ref{prop_limtame}) and hence the closure of $\tame(\pos,\C)$ by limits. 
More importantly however it turns out that this transfer assumption also  ensures the inheritance by
$\tame(\pos,\C)$ from $\C$ of both the abelian structure with  minimal resolutions (see Theorem~\ref{th_tame_ab})  and  the model structure with minimal confibrant replacements (see Theorem~\ref{th_tame_mod}). 

A priori, the transfer requirement may look too restrictive, as it is not obvious that there are meaningful examples of poset functors satisfying it. 
However, we show that there is a rich family of posets satisfying this condition:
the realizations of finite posets of dimension at most 1  (see Theorem~\ref{adfgsdfjdghn}).
These include $[0,\infty)$ and, more generally, those posets whose Hasse diagrams are trees, zig-zags, and fences (see Figure~\ref{fig_ex_dim_poset}). 

The dimension-one assumption also plays an important role in our structure theorem.
For   an abelian category  $\C$  in which minimal resolutions exist, we show that in the classical model structure on the category of non-negative chain complexes 
$\ch(\C)$ all objects admit  minimal cofibrant replacements
(see Theorem~\ref{cor:ch-min-cofbnt-replmnt}). Furthermore, we provide
a characterization of these cofibrant chain complexes whose homologies have projective dimension at most 1 (see Theorem~\ref{adgfgjhjk}). 
This is the basis for our structure theorem. 
The poset dimension requirement in our structure theorem turns out to be essential. 
In Figure~\ref{diag_the_counterexample}, we construct an example illustrating it. To do so,
 we develop a gluing technique that we present in Theorem~\ref{th_gluing}. 
This is a general method to build indecomposables in any abelian functor category, which is in itself of independent interest.

\paragraph{Related work.} 
The works \cite{realisations_posets,bcw,dwyerspalinski} define a model structure on chain complexes indexed by upper semi-lattices, non-negative reals, and finite posets, respectively. 
Our work extends those results providing a model structure to a new class of posets: realizations of finite posets of dimension at most $1$. 

Minimality is a classical notion that has been intensively studied in homotopy theory (see, e.g., \cite{deLyra,halperin,moore,quillen,sullivan1}),  in representation theory (see, e.g., \cite{auslander_reiten} and references therein), and in algebra (see, e.g., \cite{weibel} and references therein), although this list is far from being exhaustive.

Homotopy theory and homological algebra have also been the subjects of discussion in the context of applied topology. 
In \cite{BlumbergLesnick}, the authors use a model structure on the category of parametrized topological spaces to study the stability of homological invariants. 
The authors of \cite{bubemili2021} develop several results of classical homological algebra in the setting of persistence modules, i.e., functors from a poset to the category of vector spaces. 
In \cite{hess2023persistent}, the authors present a study of rational homotopy theory for parametrized chain complexes. 
All these works focus on aspects that are fascinating but orthogonal to our goals. 

The term ``tame'' has been used in the persistence theory literature with various meanings: in \cite{Cohen-Steiner2007}, it is an attribute of filtering functions over topological spaces; in \cite{scolamiero2017multidimensional}, it is a property of functor categories indexed by the poset $\mathbb{Q}^r$; in \cite{realisations_posets}, it is a property of functor categories indexed by any poset; in \cite{miller2020homological} it is an attribute of persistence modules.  
Common to these different notions is the idea that an object indexed by an infinite poset can nevertheless be described using only finite indexing data. 
Sometimes, as in \cite{Cohen-Steiner2007} and \cite{miller2020homological}, tameness also entails that the landing category is that of finite-dimensional vector spaces. 
To weaken these requirements in the case of functors with values in vector spaces, the \emph{quadrant} (or $q$-){\em tameness} in \cite{Chazal2016} only requires a finiteness condition on the ranks of the transition morphisms of persistence modules. The reason for this variability in the definition of tameness is that the term only serves the goal of describing a well-behaved class of objects in light of the problem at hand. As the problems vary, also the notion of tameness varies. 

Our goal is to provide a structure theorem for cofibrant functors in any abelian category, and minimality is instrumental in its proof and for its possible applications. 
Hence, for us, tameness is instrumental to endow our functor categories, with chain complex-valued functors being our main interest, with an abelian and a model structure. To achieve this goal, the notion of tameness we adopt here is exactly that of \cite{realisations_posets}. In particular, we consider functors landing in any abelian category. 

At the time of our writing, a preprint has appeared which also studies abelianity of functor categories, albeit not their model structures \cite{waas}. 
In that preprint, the considered functor category is a subcategory of vector space-valued functors. 
This category is the one of tame functors in the sense of \cite{miller2020homological}, which can also be characterized as finitely encodable functors in the sense of \cite{waas}, and, if one disregards that this notion specifically requires pointwise finite-dimensional vector spaces, is more general than the one we adopt here. 
Indeed, the notion of admitting a transfer at the core of our results is an example of being finitely encodable with the additional requirement that it is an adjoint. 
In \cite{waas}, it is shown that under certain hypotheses, the category of finitely encodable persistence modules is abelian. 
However, such hypotheses are not the same that we consider in the present work, as explained in more details in Remark~\ref{rem_sedwkjenow}.

As for works that consider posets of dimension 1, an application of 
the use of vector space-valued functors indexed by posets of dimension 1  is given in \cite{MR4468593}. 
There, merge trees are used to index functors which encode the evolution of homologies of connected components of sublevel sets of a function.

A gluing result for indecomposables in functor categories, related to our Theorem~\ref{th_gluing}, however valid under different assumptions, is also presented in \cite[Lemma 3.4]{BauerScoccola}.

\section{Minimality, abelianity, and model structures}
In this section, we present several categorical concepts that play an important role in this article.

\subsection{Minimality}\label{sec_mini}
Let $\C$ be a category and $f\colon X\to Y$ and $p\colon W\to Y$ be  morphisms in $\C$ with the same codomain. 
The symbol $\text{mor}_{\C\downarrow_Y}(f,p)$ is used to denote the set of morphisms $c\colon X\to W$ in $\C$ for which
$pc=f$.
Every element of $\text{mor}_{\C\downarrow_Y}(f,p)$ is called a \define{lift} of $f$ along $p$.
The  elements
of  $\text{mor}_{\C\downarrow_Y}(f,f)$
are called 
\define{endomorphisms} of $f$. A morphism $f$ is called \define{minimal}
if all its endomorphisms are isomorphisms in $\C$.

To  factorize  a morphism $f\colon X\to Y$  means to express it as a composition $f= pc$ of $c\colon X\to W$ followed by $p\colon W\to Y$.
Such a factorization is called a \define{minimal factorization}
if every morphism $\varphi$ in $\C$ making the following diagram commute is an isomorphism
\[
\begin{tikzcd}
X\arrow[r, "c"]\arrow[d, "c"'] & W\arrow[d, "p"]\\
W\arrow[r, "p"]\arrow[ur, "\varphi"] & Y
\end{tikzcd} \quad .
\]

We often identify a factorization $f = pc$ with the element $c$ in
$\text{mor}_{\C\downarrow_Y}(f,p)$.

\subsection{Abelian categories and minimal projective resolutions}\label{subs_abmin}
Consider an abelian category $\C$~\cite[Ch. VIII, Sect. 3]{maclane}.
An object $P_0$ in $\C$ is projective if, for every $f\colon P_0\to Y$ and every epimorphism $p\colon W\to Y$, $\text{mor}_{\C\downarrow_Y}(f,p)$ is non-empty.
An abelian category $\C$ has \define{enough projectives} if, for every object $X$ in $\C$, there is an epimorphism $P_0\to X$ with $P_0$ projective. 
Such an epimorphism is called a \define{(projective) cover} of $X$.
An exact sequence $\dots\to P_n\to \dots\to P_0\to X\to 0$ in $\C$ (briefly, $P\to X$) is called a \define{projective resolution} of $X$ if every $P_i$ is projective in $\C$. 
An abelian category has enough projectives if and only if each of its objects has a projective resolution.

Let $P\to X$ be a projective resolution. 
The smallest integer $n$ (if it exists) for which
$P_k=0$, for every $k>n$, is called the \define{length} of the resolution.
If such an $n$ does not exist, then the resolution is said to be infinite. 
The smallest length of all projective resolutions of $X$ is called the \define{projective dimension} of $X$ and is denoted by $\text{proj-dim}(X)$. Recall that the following equivalences hold:
\begin{itemize}
    \item $\text{proj-dim}(X)=0$ if and only if $X$ is projective.
    \item $\text{proj-dim}(X)\leq 1$ if and only if for every projective
    resolution $P\to X$,  the kernel of  $P_0\to X $ is projective.
    \item For $n\geq 2$, $\text{proj-dim}(X)\leq n$
    if and only if for every projective
    resolution $P\to X$, the kernel of the morphism $P_{n-1}\to P_{n-2}$ is projective.
\end{itemize}

If every object in $\C$ admits a projective resolution, then the supremum of the projective dimensions of objects in $\C$ is called the projective dimension
of $\C$ and is denoted by $\text{proj-dim}(\C)$. 
For example, the projective dimension of the category $\vect$ of vector spaces over a field $\mathbb{F}$ is $0$, as all its objects are projective. 
\medskip

A projective resolution (or cover) $p\colon P\to X$ is  \define{minimal}  (see Section~\ref{sec_mini}) if 
every $\varphi\colon P\to P$ such that $p\varphi=p$ is an isomorphism.
A minimal projective resolution (or cover) of $X$ is unique up to possibly a non-unique isomorphism.
If every object in $\mathcal{C}$ admits a minimal cover,
then $\mathcal{C}$ is said to satisfy the {\bf minimal cover axiom}. 
Because of this uniqueness, we will sometimes refer to the domain $P$ of a minimal projective resolution (or cover) $p\colon P\to X$ also as a minimal projective resolution (or cover) of $X$.

A projective resolution $P\to X$ is minimal if and only if the morphisms $P_0\to X$, 
$P_1\to \ker(P_0\to X)$, and
$P_k\to \ker(P_{k-1}\to P_{k-2})$, for all $k\geq 2$, are minimal projective covers. 
This is a consequence of the following basic property of minimal projective covers, which can be seen by composing $\gamma$ with $\gamma^{-1}$ and $\varphi$ with the identity: 

\begin{proposition}\label{min_cover_iso}
If $p\colon P_0\to X$ is a minimal projective cover and $\gamma$ an isomorphism of $X$, then every $\varphi\colon P_0\to P_0$ in $\C$ such that $p\varphi=\gamma p$ is also an isomorphism.   
\end{proposition}

It follows that $\C$ satisfies the minimal cover axiom if and only if all its objects admit a minimal projective resolution.
Furthermore, if $X$ admits a minimal projective resolution, then its length coincides with the projective dimension of $X$.
\medskip

A factorization  $f=pc$ in an abelian category is called a \define{projective factorization}
if $c$ is a monomorphism with projective cokernel and $p$ is an epimorphism. 
An abelian category $\C$ is said to satisfy the \define{minimal projective factorization axiom} if every morphism in $\C$ has a minimal (see Section~\ref{sec_mini}) projective factorization.

The abelian structure can be exploited to show that the minimal projective factorization axiom is equivalent to the minimal cover axiom:

\begin{proposition}\label{min_cover_equi_min_monoepi}
An abelian category $\C$ satisfies the minimal cover axiom if and only if it satisfies the minimal projective factorization axiom. 
Explicitly, when it exists, a minimal projective factorization of $f\colon X\to Y$ in $\C$ is given by $f=\pi c$ where $c= [\begin{smallmatrix} \id\\ 0\end{smallmatrix}]\colon X\to X\oplus P$, $\pi= [\begin{smallmatrix} f & p \end{smallmatrix}]\colon X\oplus P\to Y$, and $p\colon P\to Y$ is a lift of a minimal cover $P\to \coker(f)$
along the quotient $Y\to \coker(f)$. 
\end{proposition}

\begin{proof}
The sufficient condition is straightforward, as a minimal cover of an object $X$ in $\C$ can be obtained by taking a minimal projective factorization of the morphism $0\to X$. 

To prove the necessary condition, let us consider a morphism $f\colon X\to Y$ and a minimal projective cover $p'\colon P\to \coker(f)$. 
Because $P$ is projective and $Y\to \coker(f)$ is an epimorphism, we can lift $p'$ to $p\colon P\to Y$.
The morphisms $c= [\begin{smallmatrix} \id\\ 0\end{smallmatrix}]\colon X\to X\oplus P$ and $\pi= [\begin{smallmatrix} f & p \end{smallmatrix}]\colon  X\oplus P\to Y$ form a factorization of $f$. 
It is clear that $c$ is a monomorphism, $\pi$ is an epimorphism, and $\coker(c)$ is projective. 
To show that this factorization is minimal, consider an endomorphism $\varphi\colon X\oplus P\to X\oplus P$ commuting with the factorization. 
We can write $\varphi = [\begin{smallmatrix} \alpha & \beta \\ \gamma & \delta\end{smallmatrix}]$.
Since $\varphi c = c$, we have $\alpha = \id$ and $\gamma = 0$. 
Finally, by the minimality of $p'\colon P\to \coker(f)$, we have that $\delta$ is an isomorphism, and, hence, $\varphi$ is also an isomorphism. 
\end{proof}

Since 
the minimal cover axiom is equivalent to  the
minimal projective factorization axiom, in the rest of the paper we focus on the minimal cover axiom.

\subsection{Model categories and minimal cofibrant replacements}\label{subs_modcat}
A model structure on a category $\C$ is given by three distinguished classes of morphisms: weak equivalences, fibrations, and cofibrations, subject to the following axioms (for more details see~\cite{dwyerspalinski}):
\begin{itemize}
\item[(MC1)] All finite limit and colimits exist in $\C$.
\item[(MC2)] If $f$ and $g$ are morphisms in $\C$ such that $g f$ is defined, and if two of the three morphisms $f,\ g, \ g f$ are weak equivalences, then so is the third.
\item[(MC3)] If $f$ is a retract of $g$ and $g$ is a weak equivalence, fibration, or cofibration, then so is $f$.
\item[(MC4)] There exists a lift in every commutative square whose left vertical map is a cofibration, the right vertical map is a fibration, and either of them is also a weak equivalence.
\item[(MC5)] Every morphism $f$ in $\C$ can be factorized as $\pi c$, where $\pi$ is a fibration, $c$ is a cofibration, and either of them is also a weak equivalence.
\end{itemize}

(MC1) implies that every model category has the initial and the terminal objects (denoted by $\varnothing$ and $\ast$ respectively), and every morphism has a kernel and a cokernel.

A factorization $f=\pi c$ of a morphism in a model category, where $c$ is a cofibration and $\pi$ is a fibration and a weak equivalence, is called a \define{cofibrant factorization}. 
The fibration and weak equivalence $\pi\colon C\to X$  in a cofibrant factorization of $\varnothing \to X$ is called a \define{cofibrant replacement} of $X$.
If $\varnothing \to X$ is a cofibration,
the object $X$ is called \define{cofibrant}.
\medskip

If every morphism $f\colon X\to Y$ in a model category $\C$ admits a  
minimal cofibrant factorization (see Section~\ref{sec_mini}), then we say that $\C$ satisfies the \define{minimal cofibrant factorization axiom}.
The fibration and weak equivalence $\pi\colon C\to X$ in a minimal cofibrant factorization of $\varnothing \to X$ is called a \define{minimal cofibrant replacement} of $X$. 
Both minimal cofibrant factorizations and minimal cofibrant replacements are unique up to possibly not unique isomorphisms.
We suspect that, differently from what happens for minimality in abelian categories (see Proposition~\ref{min_cover_equi_min_monoepi}), the existence of a minimal cofibrant replacement for all objects in $\C$ does not imply that $\C$ satisfies the minimal cofibrant factorization axiom. 
\medskip

Cofibrant objects in a model category, similarly to projective objects in an abelian category, tend to be more manageable than arbitrary objects.
In particular, there are
relevant model categories in which all cofibrant indecomposables can be described and enumerated, in contrast to arbitrary indecomposable objects.
We, therefore, think about a minimal cofibrant replacement as a simplifying invariant (see for example~\cite{bcw}; note that, there, minimal cofibrant replacements are called minimal covers).

\section{Minimality in chain complexes}\label{sec_chain}
In this section, we recall the abelian and model structures on the category of $\ch(\C)$ of non-negative chain complexes in $\C$, for $\C$ an abelian category.  
We aim to discuss minimality in these structures.
\medskip

For a chain complex $X$, the morphism
$X_n\to X_{n-1}$ is called its ($n$-th) boundary and is denoted by $\partial_n$, or simply $\partial$.
We identify $\C$ with the full subcategory of $\ch(\C)$ given by the chain complexes which are zero in positive degrees and hence called \textbf{complexes concentrated in degree $0$}. 
A projective resolution $P\to X$ in $\C$ can be seen as an example of a morphism between the non-negative chain complex $P$ and the complex $X$ concentrated in degree $0$.
\medskip 

For $X$ in $\ch(\C)$, the symbol $S(X)$ denotes the suspension of $X$, which is a chain complex such that $S(X)_0 = 0$ and $S(X)_i = X_{i-1}$ for $i>0$,
with boundary map $S(X)_1=X_0\to 0=S(X)_0$ being the zero morphism, and the other boundary maps $S(X)_{i+1}=X_{i}\to X_{i-1}=S(X)_i$ given by
$\partial_i\colon X_{i}\to X_{i-1}$ if $i>0$.
For a natural number $n$, the symbol $S^n(X)$
denotes either $X$, if $n=0$, or the application of the suspension operation $n$ times to $X$, if $n>0$.

Let $A$ be an object in $\C$. 
The \define{sphere} in degree $n$ on   $A$ is the chain complex $S^n(A)$. 
Explicitly, $S^n(A)$ is  equal to $A$ in degree $n$, and $0$ otherwise. 
The chain complex corresponding to $A$ concentrated in degree $0$ is also denoted by $D^0(A)$ and called the \define{disk (on $A$) in degree $0$}. 
For an integer $n\geq 1$, the \define{disk (on $A$) in degree $n$}, denoted by the symbol $D^n(A)$, is by definition the chain complex which is zero in degrees different from $n$ and $n-1$ and whose $n$-th boundary is given by $\partial_n=(\text{id}\colon A\to A)$. 
When the degree does not need to be specified, even when it is 0, a disk on $A$ is simply denoted by $D(A)$. 
Note that if $n\geq 1$, then $D^n(A)=S^{n-1}(D^1(A))$.

\subsection{Minimal projective covers of chain complexes}\label{subsec:chain}
In this section, we discuss how the considered properties and constructions in the abelian category $\C$ translate when we move to the category $\ch(\C)$.
\medskip 

It is well-known (see, e.g., \cite[Th.1.2.3]{weibel}) that since $\C$ is abelian, then so is  $\ch(\C)$. 
Even though $\C$ is guaranteed to be closed only under finite direct sums,  $\ch\left(\C\right)$ is closed under direct sums of possibly infinite, but of finite type, families: 
a family of objects in $\ch(\C)$ is called \define{of finite type} if in each degree only finitely many of its elements are nonzero. 
\medskip

Let $A$ be an object in $\C$.
The disk $D^n(A)$ has the following property:
for a chain complex $X$, the restriction to the degree $n$ function, $\text{mor}_{\ch(\C)}(D^n(A), X)\to 
\text{mor}_{\C}(A, X_n)$, is a bijection.
This property implies that, if $A$ is a projective object in $\C$, then $D^n(A)$ is a projective object in $\ch(\C)$ (see also~\cite[Rem.7.9]{dwyerspalinski}). 
All finite type direct sums of such disks 
are therefore also projective objects in $\ch(\C)$.
It turns out  that all projective objects in 
$\ch(\C)$ are of this form. We refer to~\cite[Rem.7.9, Lem.7.10]{dwyerspalinski} and~\cite[Exercises in Sect.1,2]{weibel} for the proof of the following fact:

\begin{proposition}\label{prop:projective_characterization}
The following statements about an object $X$ in  $\ch(\C)$ are equivalent:
\begin{itemize}
\item $X$ is projective;
\item $H_0(X)$ and $X_n $ are projective in $\C$ for all $n\geq 0$, and $H_n(X)=0$ for all $n\geq 1$;
\item $X$ is isomorphic to $\bigoplus _{i\in I}D(A_i)$, where $A_i$ is projective  in $\C$ for all $i$ in $I$.
\end{itemize}
\end{proposition}

If $\C$ has enough projectives, then so does $\ch(\C)$. For example, a projective cover of a chain complex $X$ can be constructed by taking direct sums of disks on the projective covers of  $X_n$, for every $n\geq 0$. However, even if these are
minimal in $\C$, the obtained cover of $X$ may fail to be minimal in $\ch(\C)$. 
To obtain a minimal cover of $X$ in $\ch(\C)$ we need a slightly different construction.

\begin{proposition}\label{prop:minimal_cover}
If $\C$ is an abelian category satisfying the minimal cover axiom, then so is $\ch(\C)$. 
Explicitly, a minimal projective cover of a chain complex $X$ is isomorphic to $\bigoplus_{n\ge 0} D^{n}(P_n)$,
where $P_n$ is a minimal projective cover of $\coker(\partial_{n+1})$ in $\C$.
\end{proposition}

\begin{proof}
Let $X$ be an object in $\ch(\C)$. 
For every  $n\geq 0$, consider
a minimal projective cover
$c_n\colon P_n\to \coker(\partial_{n+1})$ in $\C$, and a lift
$\overline{c_n}\colon P_n\to X_n$ along the
quotient $X_n\to \coker(\partial_{n+1})$, which exists
since $P_n$ is projective. 
Let
$D^n(P_n)\to X$ be the unique morphism in $\ch(\C)$ which in degree $n$ is given by $\overline{c_n}$. 
We claim that
the direct sum of these morphisms
$\alpha\colon \bigoplus_{n\geq 0} D^n(P_n)\to X$
is a minimal projective cover of $X$ in $\ch(\C)$.

In each degree $n$, the morphism $\alpha$ fits into the following commutative diagram in $\C$: 
\[
\begin{tikzcd}[ampersand replacement=\&,
row sep=14pt]
\& \&P_{n+1} \arrow["{\scalemath{0.7}{\begin{bmatrix} 1 \\ 0 \end{bmatrix}}}" description, hook]{r}
\arrow[two heads]{d}{c_{n+1}} \& P_{n+1}\oplus P_{n} \arrow[two heads]{r}{\scalemath{0.7}{\begin{bmatrix} 0 & 1 \end{bmatrix}}} \arrow{dd}{\alpha_n} \& P_n \arrow[two heads]{dd}{c_n}
\\
X_{n+2}\arrow{r}{\partial_{n+2}} \&X_{n+1}\arrow[r, two heads]\arrow[rd, two heads]
\arrow[bend right=60]{drr}{\partial_{n+1}}\&
\coker(\partial_{n+2})\arrow[d, two heads]
\\ 
\& \& \image(\partial_{n+1}) \arrow[r, hook] \& X_n \arrow[r, two heads] \& \coker(\partial_{n+1})
\end{tikzcd}
\]
Since the top and bottom rows are exact,
and the external vertical morphisms are epimorphisms, by the Five Lemma $\alpha$ is an epimorphism in every degree. 
The morphism $\alpha$ is therefore a projective cover.

To prove minimality of $\alpha$, consider
an endomorphism $\Phi$ of $\bigoplus_{n\ge 0} D^{n}(P_n)$ for which
$\alpha\Phi=\alpha$. This equality means
commutativity of the following diagram for all $n\geq 0$:
\[
\begin{tikzcd}[ampersand replacement=\&, column sep=.3cm, row sep = 25pt]
\& P_{n+2}\oplus P_{n+1} \arrow["\scalemath{0.6}{\begin{bmatrix} 0 & 1 \\ 0 & 0 \end{bmatrix}}" description, pos=.3]{rrr} \arrow["\alpha_{n+1}" description, dr, two heads] \arrow[swap]{ddl}{\Phi_{n+1}=\scalemath{0.7}{\begin{bmatrix} x_{n+1} & y_{n+1} \\ z_{n+1} & w_{n+1} \end{bmatrix}}} \& \& \& P_{n+1}\oplus P_{n} \arrow["\alpha_{n}" description, dr, two heads] \arrow[crossing over, "\Phi_{n}=\scalemath{0.7}{\begin{bmatrix} x_{n} & y_{n} \\ z_{n} & w_n \end{bmatrix}}" swap,pos=.3]{ddl}
\&
\\
\& \& X_{n+1} \arrow["\partial_{n+1}" description, pos=0.2, rrr] \& \& \& X_n 
\\
P_{n+2}\oplus P_{n+1} \arrow["\alpha_{n+1}" description, urr, two heads] \arrow["\scalemath{0.7}{\begin{bmatrix} 0 & 1 \\ 0 & 0 \end{bmatrix}}" description, pos=.5]{rrr} \& \& \& P_{n+1}\oplus P_{n} \arrow["\alpha_{n}" description,urr, two heads] \arrow[from=uur, crossing over] \&
\end{tikzcd} 
\]
Using this commutativity, by direct verification, we obtain the  equalities 
$z_n=0$ and $w_{n+1}=x_n$ for all $n\geq 0$.
Thus, to show that $\Phi$ is an isomorphism, we need to prove that $w_n$ is an isomorphism for all $n\geq 0$. 
This is a consequence of
the minimality of $c_n$ and the commutativity of the following triangle which is obtained by 
taking the cokernels of the horizontal 
morphisms in the above diagram:
\[
\begin{tikzcd}[column sep=.3cm, 
row sep = 12pt]
& P_n \arrow["c_n", pos=.3, dr, two heads]
\arrow[ddl, "w_n"'] & \\
& & \coker(\partial_{n+1}) \\
P_n \arrow["c_n", swap, pos=.5, urr, two heads]
\end{tikzcd} \qedhere
\]
\end{proof}

For example, consider $X$ in $\C$, and recall that we identify $\C$ with the subcategory of chain complexes concentrated in degree 0. 
Via this identification, a projective resolution of $X$ in $\C$ is also a projective resolution of $X$ in $\ch(\C)$ (see Proposition~\ref{prop:projective_characterization}). 
The converse does not hold true. 
However, a minimal resolution of $X$ in $\ch(\C)$, according to Proposition~\ref{prop:minimal_cover}, is also a minimal projective resolution in $\C$.
This is the reason why we identify a minimal projective resolution in $\C$ with a minimal projective resolution in $\ch(\C)$.

For another example, let $\mathbb{F}$ be a field and consider the sphere $S^n(\mathbb{F})$ in  $\ch(\vect)$. 
According to Proposition~\ref{prop:minimal_cover}, a minimal projective cover of $S^n(\mathbb{F})$ is given by an epimorphism
$D^{n}(\mathbb{F})\to S^n(\mathbb{F})$.  
The kernel of this epimorphism is the sphere 
$S^{n-1}(\mathbb{F})$. We can continue this argument and obtain that  a minimal projective resolution
of $S^n(\mathbb{F})$ in  $\ch(\vect)$ is of the form:
\[
0\to D^0(\mathbb{F})\to D^1(\mathbb{F})\to D^2(\mathbb{F})\to\cdots\to D^{n}(\mathbb{F})\to S^n(\mathbb{F})
\]
Consequently, the projective dimension of  $S^n(\mathbb{F})$ is $n$.
\medskip

We conclude this section with a result about indecomposability:

\begin{proposition}\label{minprojecover_iff_ind}
Let $\C$ be an abelian category satisfying the minimal cover axiom. 
If an object $X$ in $\C$ is indecomposable, then all its minimal projective resolutions are indecomposable in $\ch(\C)$.
\end{proposition}
\begin{proof}
    Assume $X$ is indecomposable and $P\oplus P'\to X$ is a minimal projective resolution.
    Since $X$, $H_0(P\oplus P')$, and $H_0(P)\oplus H_0(P')$ are isomorphic, we can assume, by indecomposability of $X$, that $H_0(P')=0$. This means that $P\to X$ is also a projective resolution of $X$. 
    Thus by minimality $P'=0$.
\end{proof}
Note that the reverse implication of Proposition~\ref{minprojecover_iff_ind} does not hold. 
For example, a minimal projective resolution 
\[\begin{tikzcd}
    \cdots\ar[r]& 0 \ar[r]&\mathbb{Z}\ar["\cdot 6", r] & \mathbb{Z}\ar["(\mathrm{mod}\; 6)", r]  &\mathbb{Z}/6\mathbb{Z}\ar[r] & 0
\end{tikzcd}\]
of the decomposable abelian group $\mathbb{Z}/6\mathbb{Z}$ is
indecomposable as a chain complex.

\subsection{Minimal cofibrant replacements of chain complexes}
In Section~\ref{subsec:chain} we discussed how an abelian category structure on $\C$ leads to an abelian category structure on $\ch\left(\C\right)$, and how the  minimal 
cover axiom from $\C$ is inherited by $\ch\left(\C\right)$. 
In this section, we discuss how an abelian category 
structure on $\C$ leads to a model structure on $\ch\left(\C\right)$ and how the minimal cover axiom in $\C$ leads to the minimal cofibrant factorization axiom in $\ch\left(\C\right)$.
\medskip

We begin by recalling the standard model structure on $\ch\left(\C\right)$.

\begin{proposition}(\cite[Th.7.2]{dwyerspalinski})\label{prop:ch_model}
Let $\C$ be an abelian category.
The following choices of weak equivalences, fibrations, and cofibrations in the category of non-negative chain complexes $\ch(\C)$ satisfy the requirements of a model structure. 
A morphism $f\colon X\to Y$ in $\ch(\C)$ is
\begin{itemize}
    \item a weak equivalence if it is a quasi-isomorphism (induces an isomorphism in homology in all degrees);
    \item A fibration if $f_n\colon X_n\to Y_n$ is an epimorphism for every degree $n\geq 1$; 
    \item A cofibration if $f_n\colon X_n\to Y_n$ is a monomorphism such that $\coker(f_n)$ is projective in $\C$, for every degree $n\geq 0$. 
\end{itemize}
\end{proposition}

This choice of cofibrations implies that an object in $\ch(\C)$ is cofibrant if and only if it is degreewise projective. 
Thus, by Proposition~\ref{prop:projective_characterization}, all projective objects in $\ch(\C)$ are cofibrant, although the converse is not true.
As a consequence, if in $\C$ all objects are projective, then all objects in $\ch(\C)$ are cofibrant, and every object in $\ch(\C)$ can serve as its own minimal cofibrant replacement. 
\medskip

Our next goal is to explain how to obtain  minimal cofibrant factorizations in 
$\ch(\C)$ using minimal projective factorizations in $\C$.
We start with a basic construction.
Let $Y$ be a chain complex. Choose an integer $n\geq 0$ and a morphism $p\colon W\to Y_n$ in $\C$. Define 
$\pi_{p}\colon \overline{Y}\to Y$ to be the morphism in $\ch(\C)$  described by the following commutative diagram:
\[
\begin{tikzcd}
\overline{Y}\ar[swap]{d}{\pi_{p}}
\arrow[equal]{r} &\biggl ( 
\cdots \arrow[r, "\partial_{n+3}"] & Y_{n+2}\arrow[r, "\beta"]\arrow[equal]{d}
\arrow[bend left=35, "0" description]{rr}& Q \arrow[r, "\alpha"]\arrow[dr, phantom, "\scalebox{1.5}{\color{black}$\lrcorner$}", pos=.1] \arrow[d,swap, "q"] & W \arrow[r, "\partial_{n}p"]\arrow[d, "p"] & Y_{n-1}\arrow[r]\arrow[equal]{d} &\cdots \biggr )
\\
Y \arrow[equal]{r}&\biggl (\cdots \arrow[r, "\partial_{n+3}"] & Y_{n+2}\arrow[r, "\partial_{n+2}"] & Y_{n+1}\arrow[r, "\partial_{n+1}"] & Y_n \arrow[r, "\partial_{n}"] & Y_{n-1}\arrow[r] &\cdots  \biggr )
\end{tikzcd}
\]
where the indicated square ($\alpha$, $q$, $p$, $\partial_{n+1}$) is a pullback, and  $\beta\colon Y_{n+2}\to Q$ is its mediating morphism.
In particular, $\pi_{p}$ is the identity in all degrees different from $n$ and $n+1$. 
Here are the key properties of this construction:

\begin{lemma}\label{lemma_quasiiso}
    \begin{enumerate}
        \item  For every morphism  $f\colon X\to Y $ in $\ch(\C)$, the restriction to the degree $n$ function
        $\text{\rm mor}_{\ch(\C)\downarrow_Y}(f, \pi_p)\to 
        \text{\rm mor}_{\C\downarrow_{Y_n}}(f_n, p)$ is a bijection.
        \item  If $p\colon W\to Y_n$ is an epimorphism in $\C$, then
        $\pi_p\colon \overline{Y}\to Y$ is a fibration and a weak equivalence  in $\ch(\C)$.
    \end{enumerate}
\end{lemma}
\begin{proof}
\textit{1.} is a consequence of the universal property of pullbacks. 
As for \textit{2.}, assume $p$ is an epimorphism. 
Since this property is preserved by pullbacks in abelian categories, 
$\pi_p$ is also a degree-wise epimorphism and hence a fibration. 
The morphism $p$ being an epimorphism has the following further consequences. 
First, it implies the equality  $\image(\partial_n p)=\image(\partial_n)$, which gives that  $H_{n-1}(\pi_{p})$ is an isomorphism. 
Second, combined with a general property of pullbacks in abelian categories, it gives us that the induced morphisms $\coker(\alpha)\to \coker(\partial_{n+1})$ and $\ker(\alpha)\to \ker(\partial_{n+1})$ are isomorphisms. 
Consequently so are also the morphisms:
\[
\begin{tikzcd}
     \coker\left(Y_{n+2}\to \ker(\alpha)\right)\ar{d} \\
     \coker\left(Y_{n+2}\to \ker(\partial_{n+1})\right)
 \end{tikzcd}\ \ \ \ \ \text{and}\ \ \ \ \  
 \begin{tikzcd}
 \ker\left(\coker(\alpha)\to Y_{n-1}\right)\ar{d}\\
 \ker\left(\coker(\partial_{n+1})\to Y_{n-1}\right)
 \end{tikzcd}
 \]
As these cokernels and kernels coincide with respectively  the $(n+1)$-st and the $n$-th homology of the two chain complexes $\overline{Y}$ and $Y$, $H_{n+1}(\pi_{p})$ and $H_{n}(\pi_{p})$ are also isomorphisms.

Finally,
since in degrees different from  $n$ and $n+1$, the morphism $\pi_{p}$ is the identity, $H_{i}(\pi_{p})$ is an isomorphism for $i<n-1$ and $i>n+1$. 
\end{proof}

Let us choose a morphism $f\colon X\to Y$ in 
$\ch(\C)$, an integer $n\geq 0$, and a  factorization of $f_n$ in $\C$:

\[\begin{tikzcd}
    X_n\ar{r}{c}
    \ar[bend right=30, "f_n" description]{rr}& W\ar{r}{p} & Y_n
\end{tikzcd}
\]
Let us think about $c$ as an element of 
$\text{\rm mor}_{\C\downarrow_{Y_n}}(f_n, p)$, and  use Lemma~\ref{lemma_quasiiso}.1 to get the unique morphism
$\overline{c}\colon X\to \overline{Y}$ in 
$\ch(\C)$ which fits into the following factorization of $f$ and 
whose restriction in degree $n$ is the 
chosen factorization above:
\[
\begin{tikzcd}
    X\ar{r}{\overline{c}}
    \ar[bend right=30, "f" description]{rr}& \overline{Y}\ar{r}{\pi_p} & Y
\end{tikzcd}
\]

Lemma~\ref{lemma_quasiiso}.1 has several consequences. 
For example, the fact that restricting to degree $n$ is enough to determine the whole morphism allows us to construct a global inverse from the local inverse in degree $n$. 
In turn, this, together with Proposition~\ref{min_cover_iso}, gives  the following result:

\begin{corollary}\label{cor_bla}
    If $f_n= p c$ is a minimal factorization in $\C$, then
    $f=\pi_p \overline{c}$ is a minimal factorization in $\ch(\C)$.
\end{corollary}

We are now ready to prove:

\begin{theorem}\label{cor:ch-min-cofbnt-replmnt}
If $\C$ is an abelian category satisfying the minimal cover axiom, then the model structure on $\ch(\C)$
described in Proposition~\ref{prop:ch_model} satisfies the minimal cofibrant factorization axiom.
\end{theorem}

Our objective is not just to prove the theorem but to give an explicit construction of minimal cofibrant factorizations in $\ch(\C)$ using minimal projective factorizations in $\C$.

\begin{proof}
Let $f\colon X\to Y$ be a morphism in $\ch(\C)$. 
Consider the following commutative diagram, where the indicated squares (\scalebox{1.5}{$\lrcorner$}) are pullbacks and the indicated factorizations (yellow rectangles) are minimal projective factorizations in $\C$.
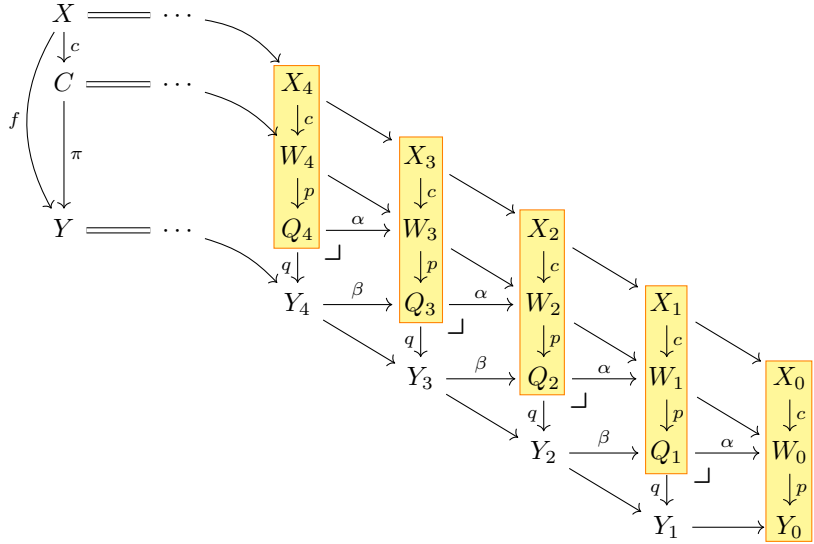
\[
\tikz[ 
overlay]{\filldraw[fill=yellow!50,draw=red!50!yellow] (3.61,2.73) rectangle (4.21,0.3);

\filldraw[fill=yellow!50,draw=red!50!yellow] (5.28,1.78) rectangle (5.85,-0.69);

\filldraw[fill=yellow!50,draw=red!50!yellow] (6.88,.81) rectangle (7.47,-1.65);

\filldraw[fill=yellow!50,draw=red!50!yellow] (8.55,-0.2) rectangle (9.1,-2.7);

\filldraw[fill=yellow!50,draw=red!50!yellow] (10.15,-1.2) rectangle (10.76,-3.6);
}
\begin{tikzcd}[row sep=12pt]
X\arrow[equal]{r}\ar{d}{c} 
\ar[swap, bend right= 30, "f"]{ddd}
&\cdots \ar[bend left=20]{dr}
\\
C\arrow[equal]{r}\ar{dd}{\pi}&\cdots\ar[bend left=15]{dr}   & X_4\ar{d}{c}\ar{dr}
\\
& & 
W_4\ar{d}{p}\ar{dr}  & X_3\ar{d}{c}\ar{dr}
\\
Y\arrow[equal]{r}&\cdots  \ar[bend left=15]{dr}& 
Q_4\ar{r}{\alpha}\ar[swap]{d}{q}
\arrow[dr, phantom, "\scalebox{1.5}{\color{black}$\lrcorner$}", pos=.1]
&W_3\ar{d}{p} \ar{dr}& X_2\ar{d}{c}\ar{dr}
\\
& & Y_4\ar{r}{\beta}\ar{dr} & Q_3\ar{r}{\alpha}\ar[swap]{d}{q}
\arrow[dr, phantom, "\scalebox{1.5}{\color{black}$\lrcorner$}", pos=.1]
&W_2\ar{d}{p} \ar{dr}& X_1\ar{d}{c}\ar{dr}
\\
&& & Y_3\ar{r}{\beta}\ar{dr} & Q_2\ar{r}{\alpha}\ar[swap]{d}{q}
\arrow[dr, phantom, "\scalebox{1.5}{\color{black}$\lrcorner$}", pos=.1] & W_1\ar{d}{p} \ar{dr}& X_0\ar{d}{c}
\\
& & &  &
Y_2\ar{r}{\beta}\arrow{dr} & Q_1\ar{r}{\alpha}
\ar[swap]{d}{q}
\arrow[dr, phantom, "\scalebox{1.5}{\color{black}$\lrcorner$}", pos=.1]
& W_0\ar{d}{p}
\\
& 
 & &  &
& Y_1\ar{r} & Y_0
\end{tikzcd}    
\]

We claim that the  obtained factorization $f=\pi c$ in $\ch(\C)$
is a minimal cofibrant factorization.
The minimality and the fact that $\pi$ is a fibration and a weak equivalence follow respectively from Corollary~\ref{cor_bla} and Lemma~\ref{lemma_quasiiso}.2.
Since for every $n\geq 0$, the  morphism $c\colon X_n \to W_n$ is
a monomorphism with projective cokernel, $c$ is a cofibration in 
$\ch(\C)$.
\end{proof}

We conclude the section by observing that, in analogy to what happens for minimal projective covers (see the end of Section~\ref{subsec:chain}), the decomposability of an object in $\ch(\C)$ does not imply the decomposability of its minimal cofibrant replacements in $\ch(\C)$.
For example, taking $\C$ to be the category of abelian groups, a minimal cofibrant replacement of the chain complex $S^0(\mathbb{Z}/6\mathbb{Z})$ is 
\[
\begin{tikzcd}[ampersand replacement=\&]
\cdots \ar[r] \&0 \arrow[r] \arrow[d] \& \mathbb{Z} \arrow[d] \arrow[r,"\cdot 6"] \&  \mathbb{Z} \arrow[d, "\text{(mod 6)}"]
\\
\cdots \arrow[r] \& 0 \arrow[r] \& 0 \arrow[r] \&  \mathbb{Z}/6\mathbb{Z}
\end{tikzcd}
\]

However, in contrast to what happens for minimal projective covers (see Proposition~\ref{minprojecover_iff_ind}), the indecomposability of an object does not imply the indecomposability of its minimal cofibrant replacements.
To see this, consider the abelian functor category $\Fun(0\leq 1 \leq 2, \ch(\vect))$ and take the following left-hand object there, where the vertical maps are the boundary maps:
\[
\begin{tikzcd}[ampersand replacement=\&]
0 \arrow[r] \arrow[d] \& \mathbb{F} \arrow[d, "\id"] \arrow[r,"\id"] \&  \mathbb{F} \arrow[d]
\\
\mathbb{F} \arrow[r, "\id"] \& \mathbb{F} \arrow[r] \&  0
\end{tikzcd}
\hspace{2cm}
\begin{tikzcd}[ampersand replacement=\&]
0 \arrow[r] \arrow[d] \& \mathbb{F} \arrow[d, "\id"] \arrow{r}{\scalemath{0.7}{\begin{bmatrix} 1 \\ 0 \end{bmatrix}}} \&  \mathbb{F}^2 \arrow{d}{\scalemath{0.7}{\begin{bmatrix} 1 & 0 \end{bmatrix}}}
\\
\mathbb{F} \arrow[r, "\id"] \&  \mathbb{F} \arrow[r] \& \mathbb{F}
\end{tikzcd} 
\]
The object is indecomposable, but the minimal cofibrant replacement on the right-hand side is decomposable.

\subsection{Structure theorem for cofibrant objects}
Recall that our model structure of choice on the category $\ch(\C)$ is the one described in Proposition~\ref{prop:ch_model}. 
In particular, an object $X$ in $\ch(\C)$ is cofibrant if and only if $X_n$ is projective for all $n\geq 0$.
\medskip

The aim of this section is to prove:

\begin{theorem}\label{adgfgjhjk}
    Let $\C$ be an abelian category satisfying the minimal cover axiom. 
    Assume $X$ is a cofibrant object in $\ch(\C)$, whose homology in every degree has projective dimension at most 1. Then $X$ 
    is isomorphic to
    \[\bigoplus_{n\geq 0}S^n(P[n]) \oplus D^{n+1}(Y_n)\]
where, for every $n\geq 0$, the chain complex  $P[n]$ is a minimal projective resolution of $H_n(X)$ and $Y_n$ is a projective object in $\C$.
\end{theorem}

From Theorem~\ref{adgfgjhjk} and recalling that an object in $\ch(\C)$ and its minimal cofibrant replacements have the same homologies, we get an explicit description of a minimal cofibrant replacement of a chain complex whose homologies have projective dimension at most $1$.

\begin{corollary}\label{sdgshj}
Let $\C$ be an abelian category satisfying the minimal cover axiom.
Assume $X$ in $\ch(\C)$ is such that 
its homology in every degree has projective dimension at most~1. 
Then a minimal cofibrant replacement of $X$ in $\ch(\C)$ is isomorphic to
\[
\bigoplus_{n\geq 0}S^n(P[n]) \oplus 
D^{n+1}(Y_n)
\]
where, for every $n$, the chain complex $P[n]$ is a minimal projective resolution of $H_n(X)$ and $Y_n$ is a projective object in $\C$.
\end{corollary}

These results hinge on the fact that the minimal projective resolutions of the homologies have at most length 1. 
This firstly means that such resolutions are non-zero in at most two consecutive degrees and the boundary morphism between them is a monomorphism. 
These properties are crucial to prove Theorem~\ref{adgfgjhjk} (see Diagram (\ref{diam_decomposition})).
Therefore, we cannot directly extend our arguments to objects whose homologies have higher projective dimensions. 
However, this is not a shortcoming of the method: in Section~\ref{sec_gluing} we present an indecomposable object in
$\ch(\mathcal{F})$ that is nonzero in four consecutive degrees, where $\mathcal{F}$ is an abelian category with projective dimension $2$ (see Figure~\ref{diag_the_counterexample}).

Secondly, we can describe the indecomposable cofibrant objects in $\ch(\C)$ whose homologies have projective dimension at most $1$:

\begin{proposition}
Let $\C$ be an abelian category satisfying the minimal cover axiom. Assume  $X$ is a cofibrant object in  $\ch(\C)$ whose homology in every degree has projective dimension at most 1.
Then $X$ is indecomposable if and only if one of the following holds:
\begin{itemize}
    \item  The homology of $X$ in every  degree is trivial, in which case $X$ is isomorphic to $D^n(Y)$ for some indecomposable projective object $Y$ in $\C$, 
    \item The homology of $X$ is nontrivial only in one degree, say $n$,  in which case $H_n(X)$ is indecomposable and $X$ is isomorphic to $S^n(P)$ where $P$ is a minimal projective resolution of $H_n(X)$.
 \end{itemize}
\end{proposition}

\begin{proof}
If the homology of $X$ is trivial in all degrees and $X$ is indecomposable, then, according to Corollary~\ref{sdgshj},
$X$ has to be isomorphic to $D^{n+1}(Y)$ for some projective $Y$ in $\C$
which necessarily has to be indecomposable. 
Moreover, for every projective and indecomposable $Y$ in $\C$, the chain complex $D^{n+1}(Y)$ is indecomposable.

If the homology of $X$ is nontrivial and $X$ is indecomposable,  then, according to Corollary~\ref{sdgshj},
the homology of $X$ has to be nontrivial only in one degree, say $n$, and $X$ should be  isomorphic to  $S^n(P)$, where $P$ is a minimal projective resolution of $H_n(X)$.
If $H_n(X)$ were not indecomposable, then a minimal projective resolution of its nontrivial summand would be a nontrivial summand of $S^n(P)$, contradicting the indecomposability of $X$.
Thus $H_n(X)$ is indecomposable. 
The indecomposability of $H_n(X)$ in $\C$ and the minimality of its projective resolution $P$ imply the indecomposability of $P$ as a chain complex. 
Thus any of its suspensions $S^n(P)$ is also an indecomposable chain complex.
\end{proof}

To prove Theorem~\ref{adgfgjhjk}, we start with

\begin{lemma}\label{fhdghkltjkdt}
    Under the assumptions of Theorem~\ref{adgfgjhjk}, let $m\geq 0$ be a natural number such that $X_i=0$ for $i<m$. 
    Then $X$ is isomorphic to
    $S^m(P[m])\oplus X'$ for some  cofibrant object $X'$ in $\ch(\C)$ for which $H_m(X')=0$.
\end{lemma}

\begin{proof}
Consider the following commutative diagram in $\C$ 
\begin{equation}\label{diam_decomposition}
\begin{tikzcd}
0\ar{r}\ar{d} & X_{m+2}\ar{r}\ar{d}[swap]{\partial_{m+2}} & 0\ar{d}
\\
P[m]_1\ar{r}{s_1}\ar[hook]{d}[swap]{d} & X_{m+1}\ar{r}{p_1}\ar{d}[swap]{\partial_{m+1}} & P[m]_1\ar[hook]{d}{d}
\\
P[m]_0\ar{r}{s_0}\ar[two heads]{dr} & X_{m}\ar{r}{p_0}\ar[two heads]{d} & P[m]_0\ar[two heads]{dl}
\\
& H_m(X)
\end{tikzcd} 
\end{equation}
where:
\begin{itemize}
    \item the minimal projective resolution $P[m]$ of $H_m(X)$ has length at most 1 by hypothesis;
    \item $s_0$ is a lift of the cover $P[m]_0\to H_m(X)$ along
    the quotient $X_m\to H_m(X)$ which exists since $P[m]_0$ is projective;
    \item $p_0$ is a lift of the quotient $X_m\to H_m(X)$ along the cover 
    $P[m]_0\to H_m(X)$ which exists since $X_m$ is projective;
    \item $s_1$ is a lift of $s_0d$ along $\partial_{m+1}$ which exists since $P[m]_1$ is projective and the sequence $X_{m+1}\to X_m\to H_m(X)\to 0$
    is exact;
    \item $p_1$ is a lift of $p_0\partial_{m+1}$ along $d$ which exists
    since $X_{m+1}$  is projective and the sequence  $P[m]_1\to P[m]_0\to H_m(X)\to 0$ is exact;
    \item The commutativity of the top right square follows from $d$ being a monomorphism and $\partial_{m+1}\partial_{m+2}=0$;
    \item The compositions $p_1s_1$ and $p_0s_0$ are isomorphisms due to the fact that $P[m]\to H_m(X)$ is a minimal resolution.
\end{itemize}
The commutativity of this diagram and the fact that $p_1s_1$ and $p_0s_0$ are isomorphisms imply  $S^m(P[m])$ is a direct summand of $X$. Consequently $X$ is isomorphic to $S^m(P[m])\oplus \text{ker}(p\colon X\to S^m(P[m]))$.
\end{proof}

\begin{lemma}\label{aGTHFHKGHJKGJDFDH}
    Under the assumptions of Theorem~\ref{adgfgjhjk}, let $m\geq 0$ be a natural number such that $X_i=0$ for $i<m$ and $H_m(X)=0$. 
    Then $X$ is isomorphic to $D^{m+1}(X_m)\oplus X'$ for some cofibrant object $X'$ in $\ch(\C)$ for which $ X'_i=0$ for $i\leq m$.
\end{lemma}

\begin{proof}
Consider the following commutative diagram in $\C$:
\[
\begin{tikzcd}
0\ar{r}\ar{d} & X_{m+2}\ar{r}\ar{d}[swap]{\partial_{m+2}} & 0\ar{d}\\
X_m\ar{r}{s_0}\ar{d}[swap]{\text{id}} & X_{m+1}\ar{r}{\partial_{m+1}}\ar{d}[swap]{\partial_{m+1}} & X_m\ar{d}{\text{id}}\\
X_m\ar[equal, r] & X_m\ar[equal, r] & X_m
\end{tikzcd} 
\]
where $s_0$ is a lift along $\partial_{m+1}$ which exists since $X_m$ is projective and $\partial_{m+1}$ is an epimorphism.  
Commutativity of this diagram shows that $D^{m+1}(X_m)$ is a direct summand of $X$. It follows that $X$ is isomorphic to $D^{m+1}(X_m)\oplus X'$
where $X'$ is a cofibrant chain complex with 
$X'_i=0$ for $i\leq m$.
\end{proof}

\begin{proof}[Proof of Theorem~\ref{adgfgjhjk}]
By applying Lemmas~\ref{fhdghkltjkdt} and~\ref{aGTHFHKGHJKGJDFDH} 
to $X$ with $m=0$, we obtain that $X$ is isomorphic to $P[0]\oplus D^1(Y_0)\oplus X'$ for some projective object $Y_0$ in $\C$ and a cofibrant chain complex $X'$ with $X'_0=0$. 
We can then continue and apply Lemmas~\ref{fhdghkltjkdt} and~\ref{aGTHFHKGHJKGJDFDH} to $X'$ with $m=1$ to get that $X'$ is isomorphic to
$S(P[1])\oplus  D^2(Y_1)\oplus X''$ for some projective object $Y_1$ in $\C$ and a cofibrant chain complex  $X''$ with $X''_i=0$ for $i\leq 1$.
This gives an isomorphism between $X$ and 
$P[0]\oplus D^1(Y_0)\oplus S(P[1])\oplus  D^2(Y_1)\oplus X''$.
By continuing this process by induction we obtain the desired description of $X$
as  $\bigoplus_{n\geq 0}S^n(P[n]) \oplus 
D^{n+1}(Y_n)$.
\end{proof}

\section{Tame functors}\label{sec_tameness}
Throughout this section, $\C$ is assumed to be a category closed under finite colimits. 
In particular $\C$ has an initial object 
$\varnothing$. 

This assumption guarantees that, for every poset functor $\alpha \colon \D\to  \pos$ 
with finite $\D$, the precomposition with $\alpha$ functor, $(-)^\alpha\colon \Fun(\pos, \C)\to \Fun(\D, \C)$, has a left adjoint.
This left adjoint is called 
\define{left Kan extension} along $\alpha$
and is denoted by $\alpha^k\colon \Fun(\D, \C)\to \Fun(\pos, \C)$.
In detail, for a functor $Y\colon\D\to \C$ and  an object  $x$ in $\pos$, the value $\alpha^k Y(x)$ is isomorphic to $ \colim_{\alpha\le x}Y$, where $\alpha\le x\coloneqq\{d\in \D\mid \alpha(d)\le x\}$. This colimit exists since $\alpha\le x$ is finite.
We are also going to use the notation $\pos\le x\coloneqq\{q\in \pos\mid q\le x\}$, and, mutatis mutandis, for the strict inequalities. 
The term a \define{full subposet} of $\pos$ refers to a subset $\D\subset \pos$ with the induced relation:
$x\leq y$ in $\D$ if and only if $x\leq y$ in $\pos$.
\medskip

For an example of left Kan extension, consider the inclusion $\alpha\colon\{x\}\subset \pos$.
A functor $ \{x\}\to \C$ is simply  a choice of an object $A$ in $\C$. 
Its left Kan extension along $\alpha$ is denoted 
by $A(x,-)\colon \pos\to\C$ and called a \define{homogenous free functor} (in degree $x$). Explicitly, $A(x,q)$ is isomorphic to either $\varnothing$, if $x\not<q$, or to 
$A$ if $x\le q$. 
Homogeneous free functors are examples of tame functors:

\begin{definition}\label{def_tame}
Let $\pos$ be a poset and $\C$ a category closed under finite colimits. 
A functor $X\colon \pos\to \C$ is  \define{tame} if 
it is isomorphic to
$\alpha^kY$ for some $\alpha\colon \D\to\pos$ with $\D$ finite, and some
$Y\colon \D\to\C$. 
In this case $X$ is also said to be \define{discretized} by $\alpha$, and $Y$ is called its \define{discretization} (along $\alpha$). The full subcategory of tame functors in $\Fun(\pos, \C)$ is denoted by $\tame(\pos, \C)$.
\end{definition}

If $\pos$ is finite, then every functor indexed by $\pos$ is tame.
Thus tameness is a restrictive condition only for functors indexed by infinite posets.  
\medskip

We now state some properties of left Kan extensions and discretizations of tame functors that will be fundamental in the rest of the work. 

\begin{proposition}\label{prop_full_lke}
Let $\C$ be a category closed under finite colimits.
\begin{enumerate}
\item If $X$ in $\tame(\pos, \C)$ is discretized by $\alpha\colon \D\to \pos$, then it is also discretized by any finite full subposet $\D'\subset \pos$  such that $\alpha(\D)\subset \D'$.
\item Let $\{X_i\}_{i=1}^n$ be a finite collection of functors in $\tame(\pos, \C)$. 
There exists a finite full subposet $\D\subset \pos$ that discretizes every $X_i$. 
\item If $\alpha \colon \D\subset \pos$ is a finite full subposet inclusion, then, for every functor $Y\colon \D\to \C$, the natural transformation $Y\to (\alpha^k Y)\alpha $ (adjoint to the identity) is an isomorphism.

\item If $X\colon \pos\to \C$ is tame, then, for every 
finite full subposet inclusion $\alpha \colon \D \subset\pos$  that discretizes $X$, the natural transformation $\alpha^k (X\alpha)\to X$ (adjoint to the identity) is an isomorphism.
\item A functor $X\colon \pos\to \C$ is tame if and only if there exists a finite full subposet inclusion $\alpha \colon \D \subset\pos$ for which the natural transformation $\alpha^k (X\alpha)\to X$ (adjoint to the identity) is an isomorphism. 
\item 
A morphism $\varphi\colon X\to X'$  in $\tame(\pos, \C)$ is an isomorphism if and only if there exists a finite full subposet inclusion $\alpha\colon \D\subset \pos$ discretizing both $X$ and $X'$ such that $\varphi^\alpha$ is an isomorphism in $\Fun(\D, \C)$. Equivalently, if and only if for every finite full subposet inclusion $\alpha\colon \D\subset \pos$ discretizing both $X$ and $X'$, the restriction $\varphi^\alpha$ is an isomorphism in $\Fun(\D, \C)$. 
\end{enumerate}
\end{proposition}

\begin{proof}
Statement 1 is a consequence of the fact that left Kan extensions commute with compositions. 

According to statement 1, 
the full subposet $\bigcup_{i=1}^{n}\alpha_i(\D_i)\subset \D$ 
discretizes all the functors $X_i$, proving statement \textit{2}.

The isomorphism in statement 3 between $Y$ and $(\alpha^k Y)\alpha$ can be checked by writing the left Kan extension explicitly as a colimit (see the beginning of the section) and using the fact that $\alpha$ is full inclusion.

Statement 4 follows directly from 3, applied to $Y$ for which $X$ is isomorphic to $\alpha^kY$.

The condition in statement 5 is sufficient by the definition of a tame functor. 
That it is necessary is a consequence of statements 1 and 4.

To show statement 6, choose a finite full subposet inclusion $\alpha\colon \D\subset \pos$ discretizing both $X$ and $X'$ (statement~2), and consider the following commutative square where the vertical arrows represent  isomorphisms according to 4:
\[
\begin{tikzcd}
\alpha^k (X\alpha)\arrow["\alpha^k(\varphi^\alpha)"]{r}\arrow[swap, "\cong"]{d} & \alpha^k (X'\alpha)\arrow["\cong"]{d}\\
X \arrow[swap, "\varphi"]{r} & X' 
\end{tikzcd}
\]
Thus, by functoriality of  $\alpha^k$, the natural transformation
$\varphi$ is an isomorphism if and only if $\varphi^\alpha$ is an isomorphism. Furthermore, if one restriction $\varphi^\alpha$ is an isomorphism, then, since $\varphi$ is an isomorphism, every other of its restrictions is also an isomorphism.
\end{proof}

Thanks to Proposition~\ref{prop_full_lke}.2, we can talk about a \define{common discretization} of a finite collection of tame functors.
\medskip

Proposition~\ref{prop_full_lke}.2, together with the fact that left Kan extensions commute with colimits, implies that, for any functor $\phi\colon I\to \Fun(\pos, \C)$ indexed by a finite category $I$, if the values of
$\phi$ are tame, then $\text{colim}_I\phi$ is also tame. Consequently, the category 
$\tame(\pos, \C)$ is closed under finite colimits and the inclusion
$\tame(\pos, \C)\subset \Fun(\pos, \C)$
commutes with finite colimits.
\medskip

\begin{remark}\label{deocnkwopewc}
It turns out that, if $\C$ is closed under finite limits, then the inclusion
$\tame(\pos, \C)\subset \Fun(\pos, \C)$
also commutes with finite limits. 
To see this, let $\varphi\colon I\to\tame(\pos, \C)$ 
be a functor indexed by a finite category $I$.
Assume that its limit in $\tame(\pos, \C)$ exists and is denoted by 
$\text{\em lim}_I\varphi$.
Let us denote the limit of $\varphi$ in 
$\Fun(\pos, \C)$ by $\lim_I\varphi$.
Recall that $\lim_I\varphi$ is formed objectwise:  $(\lim_I\varphi)(q)$ and $ \lim_I(\varphi(q))$ are isomorphic for all $q$ in $\pos$.
For every object $A$ in $\C$ and every $x$ in $\pos$, since $A(x,-)$ is tame,
by the universal property of limits we have natural bijections between the following sets:

\begin{equation*}\label{eq:limit_commutativity}
\begin{tikzcd}[row sep=12pt, column sep = 12pt]
\text{mor}_{\C}\left(A, (\text{\em lim}_I \varphi)(x)\right)\arrow[d, phantom, sloped, "\cong"]  & & 
\text{mor}_{\C}\left(A, (\text{lim}_I \varphi)(x)\right)\arrow[d, phantom, sloped, "\cong"] 
\\
\Nat(A(x,-), \text{\em lim}_I \varphi)
\arrow[r, phantom, sloped, "\cong"] &
\text{lim}_I\Nat(A(x,-), \varphi)
\arrow[r, phantom, sloped, "\cong"] &
\Nat(A(x,-), \text{lim}_I \varphi)
\end{tikzcd}
\end{equation*}
We can now use the Yoneda lemma to conclude that 
the natural transformation $\text{lim}_I \varphi\to  \text{\em lim}_I \varphi$ is an isomorphism.
\end{remark}
\medskip

The fact that the inclusion 
$\tame(\pos, \C)\subset \Fun(\pos, \C)$ commutes with finite limits does not guarantee  that the category $\tame(\pos, \C)$ is closed under finite limits. 
For this to happen, additional assumptions are necessary.
In the next section, we are going to formulate such assumptions.

\subsection{Transfer}\label{S_transfer}
We aim to identify some posets $\pos$ for which  $\tame(\pos, \C)$ inherits an abelian or a model structure from $\C$. 
The necessary and, as it turns out, often a sufficient requirement for this to happen is that $\tame(\pos, \C)$ is closed under finite limits. To ensure that tameness is preserved by finite limits,  we need a way of constructing left Kan extensions that is compatible with limits. For this, we are going to use transfers.
\medskip

The symbol $\D_\ast$ denotes the poset obtained from a poset $\D$ by adding a global minimum: 
$\D_\ast\coloneqq\D\cup \{-\infty\}$ with the following poset relation:
$-\infty<d$, for every $d$ in $\D$, and  $d\le d'$ in $\D$  if and only if  $d\le d'$ in $\D_\ast$.  
Analogously, for a poset functor $\alpha\colon\D\to \pos$, the symbol $\alpha_\ast\colon \D_\ast\to \pos_\ast$ denotes the functor for which  $\alpha_\ast(d)=\alpha(d)$, for all $d$ in $\D$, and $\alpha_*(-\infty)=-\infty$.
Similarly, every functor $X\colon\D\to\C$ extends to a functor denoted by the same
symbol $X\colon\D_\ast\to \C$ 
whose restriction to $\D\subset \D_\ast$ is $X$ and which maps $-\infty$ to
$\varnothing$.
This extension procedure induces a bijection between  the set of all  functors $\Fun(\D, \C)$ and the set 
$\Fun_\ast(\D_\ast, \C)$ of 
functors $Y\colon \D_\ast\to \C$ which map $-\infty$ in $\D_\ast$ to the initial object $\varnothing$ in $\C$. 
We use this bijection to identify the sets $\Fun_\ast(\D_\ast, \C)$ and $\Fun(\D, \C)$ in what follows.

\begin{definition}\label{def:transfer}
The \define{transfer} of a poset functor
$\alpha\colon\D\to \pos$
is the poset functor
$\alpha^!\colon \pos_\ast\to \D_\ast$
that is right adjoint to 
$\alpha_\ast\colon\D_\ast\to \pos_\ast$.
If such a right adjoint to $\alpha_\ast$ exists, then $\alpha$ is said to 
\define{admit transfer}.
\end{definition}

Recall that a pair of adjoint functors between posets is also called a Galois connection~\cite[Ch.IV, Sect.5, Th.1]{maclane}. 
Galois connections appear in many different contexts, including persistence theory: in~\cite{gulen2022galois}, Galois connections play a central role in talking about metrics and functoriality, while in \cite{bubenik2020comparison,kim2025interleaving,patel2025mobius} they enter more tangentially.
Although this is a well-studied and described notion, to keep the section self-contained, we will present arguments to explain some basic properties of transfers. 

If $\alpha\colon\D\to \pos$ has $\beta\colon \pos\to \D$ as the  right adjoint, then
$\beta_\ast\colon \pos_\ast\to \D_\ast$
is the right adjoint to $\alpha_\ast$ and hence it is the transfer of $\alpha$.
However,
admitting a transfer does not guarantee
that $\alpha$ is a left adjoint.
As a simple example, consider the inclusion $\alpha \colon \{b\}\subset 
\{a\leq b\}$.
If $\alpha$ were a left adjoint, 
the sets
$\text{mor}_{\{a\leq b\}}\left(b,a\right)$ and 
$\text{mor}_{\{b\}}\left(b,b\right)$ would have the same number of elements, but the first set is empty and the second is not. 
Nevertheless, $\alpha_\ast$ 
is a left adjoint. Its right adjoint 
(the transfer of $\alpha$) is given by the functor 
$\{-\infty\leq a\leq 
b\}\to \{-\infty\leq b\}$
mapping $-\infty$ and $a$ to $-\infty$, and $b$ to $b$.
\medskip

Here is a basic criterion for when a functor $\alpha^!\colon\pos_\ast\to \D_\ast$ is the transfer of
$\alpha\colon \D\to\pos$.

\begin{proposition}\label{ineq_transfer}
Let $\alpha\colon\D\to \pos$ be a poset functor. A functor $\alpha^!\colon\pos_\ast\to \D_\ast$ is the transfer of $\alpha$ if and only if 
$d\le \alpha^!\alpha_\ast(d)$ for all $d$ in $\D_\ast$, and $\alpha_\ast\alpha^!(q)\le q$ for all $q$ in $\pos_\ast$.
\end{proposition}

\begin{proof}
If $\alpha^!$ is the right adjoint to $\alpha_\ast$, then $\alpha_\ast(d)=\alpha_\ast(d)$ implies $d\le \alpha^!\alpha_\ast(d)$, and, since $\alpha_\ast$ is the left adjoint to $\alpha^!$,  then $\alpha^!(q)=\alpha^!(q)$ implies $\alpha_\ast\alpha^!(q)\le q$.

Assume the relations
$d\le \alpha^!\alpha_\ast(d)$ and 
$\alpha_\ast\alpha^!(q)\le q$ hold.
Then 
$d\le \alpha^!(q)$ implies $\alpha_\ast(d)\le \alpha_\ast\alpha^!(q)\le q$, and hence $\alpha_\ast(d)\le q$.
Similarly,  $\alpha_\ast(d)\le q$ yields
$d\le \alpha^!\alpha_\ast(d)\le \alpha^!(q)$, and hence $d\leq \alpha^!(q)$. 
\end{proof}

\begin{corollary}\label{sdgadfgh}
If $\alpha^!\colon \pos_\ast\to \D_\ast$ is the transfer of $\alpha\colon \D\to \pos$, then
\[\alpha^!(q)\text{ is }\begin{cases}
    -\infty & \text{if }q=-\infty\\
    -\infty &  \text{if }q\in \pos \text{ and }\alpha\leq q\text{ is empty}\\
    \text{the terminal object of }
    \alpha\leq q & \text{if }q\in \pos \text{ and }\alpha\leq q\text{ is non-empty}
\end{cases}\]
\end{corollary}
\begin{proof}
The relation $\alpha_{\ast}\alpha^!(q)\leq q$ (see Proposition~\ref{ineq_transfer}) implies that
if $\alpha^!(q)$ is in $\D\subset \D_\ast$ then $q\not = -\infty $ and
$\alpha^!(q)$ is in $\alpha\leq  q$.
Moreover, for every $d$ in $\alpha\leq  q$, combining $d\leq \alpha^!\alpha(d)$ (see again Proposition~\ref{ineq_transfer}) with
$\alpha^!$ applied to $\alpha(d)\leq q$ yields:
$d\leq \alpha^!\alpha(d)\leq \alpha^!(q)$, which means  $\alpha^!(q)$ is the terminal object in $\alpha\leq  q$.
\end{proof}

By Corollary~\ref{sdgadfgh}, a transfer maps $-\infty$ to $-\infty$. The precomposition with
$\alpha^!\colon\pos_\ast\to \D_\ast$  induces therefore a functor 
$(-)^{\alpha^!}\colon \Fun_\ast(\D_\ast, \C)\to \Fun_\ast(\pos_\ast, \C)$. Since 
$\Fun_\ast(\D_\ast, \C)$ and $\Fun_\ast(\pos_\ast, \C)$ are identified with
$\Fun(\D, \C)$ and $\Fun(\pos, \C)$, respectively,
we consider $(-)^{\alpha^!}$  as a functor $\Fun(\D, \C)\to \Fun(\pos, \C)$. 
Although via this identification $(-)^{\alpha^!}$ may fail to be the precomposition with 
any functor $\pos\to\D$, it inherits the following property of functors formed by precompositions: if $\C$ is closed under (finite) limits, then $(-)^{\alpha^!}\colon \Fun(\D, \C)\to \Fun(\pos, \C)$ commutes  
with (finite) limits.

\begin{corollary}\label{dsgdfh}
Let $\C$ be a category closed under finite colimits and $\alpha\colon\D\to \pos$ be a poset functor with finite $\D$.
If $\alpha^!\colon \pos_\ast\to \D_\ast$ is the transfer of $\alpha\colon \D\to \pos$, then
\begin{enumerate}
\item The functor $(-)^{\alpha^!}$ is naturally isomorphic to 
$ \alpha^k\colon \Fun(\D, \C)\to \Fun(\pos, \C)$.
\item If $\C$ is closed under (finite) limits, then $ \alpha^k\colon \Fun(\D, \C)\to \Fun(\pos, \C)$ commutes with
(finite) limits.
\item The full subposet inclusion $\alpha(\D)\subset\pos$ admits a transfer.
\end{enumerate}
\end{corollary}

\begin{proof}
The first statement follows from Corollary~\ref{sdgadfgh} and the formula for left Kan extensions recalled at the beginning of Section~\ref{sec_tameness}, identifying the values of the left Kan extension $\alpha^k Y(x)$ with $\colim_{\alpha\le x}Y$.
The second statement is obtained from the first one by observing that precompositions commute with finite limits whenever they exist.

To show the third part just note that
the criterion in Proposition~\ref{ineq_transfer} for $\alpha^!$ to be the transfer of $\alpha$
directly implies this criterion for the following 
composition to be the transfer of 
$\alpha(\D)\subset\pos$:
\[
\begin{tikzcd}
    Q_\ast\ar{r}{\alpha^!} & \D_\ast\ar{r}{\alpha_\ast} & \alpha(\D)_\ast
\end{tikzcd}\qedhere
\]
\end{proof}

Corollary~\ref{dsgdfh}, together with Proposition~\ref{prop_full_lke}.1, allows us to restrict our attention to tame functors discretized along an inclusion of a finite full subposet of the indexing poset. 
We can now prove that, under suitable hypotheses, tameness commutes with finite limits. 

\begin{proposition}\label{prop_limtame}
Let $\C$ be a category closed under finite colimits and limits, $I$ a finite category, $\pos$ a poset, and $\phi\colon I\to
\Fun(\pos,\C)$  a functor. 
Assume there is a poset functor $\alpha\colon\D\to \pos$ with finite $\D$, admitting a transfer, and which discretizes $\phi(i)$ for all $i$ in $I$.
Then $\text{\rm lim}_I\phi$ is tame.
\end{proposition}

\begin{proof}
According to Corollary~\ref{dsgdfh}.3, the inclusion of the full subposet $\alpha(\D)\subset \pos$ satisfies the same assumptions as $\alpha$. 
We can thus assume that $\alpha$ is an inclusion of a full subposet $\D\subset \pos$. 
 Under this assumption, by Proposition~\ref{prop_full_lke}.4, the natural transformation $\alpha^k(\phi(i)\alpha)\to \phi(i)$ is an isomorphism, for every $i$. 
Since $\alpha$ admits a transfer, we can apply Corollary~\ref{dsgdfh} and obtain  isomorphisms $\alpha^k((\text{lim}_I\phi)\alpha)\to\lim_I (\alpha^k(\phi(-)\alpha))\to  \text{lim}_I\phi $.
Applying Proposition~\ref{prop_full_lke}.5, we can conclude that $\text{lim}_I\phi
$ is tame.
\end{proof}

We finish this section with a fundamental property of tame functors discretized by functors admitting a transfer.

\begin{corollary}\label{cor_closedcolim}
Let $\C$ be a category closed under finite colimits and limits. Assume  $\pos$ is a poset for which every tame functor $X\colon \pos\to\C$
can be discretized by a functor admitting a transfer.
Then $\tame(\pos, \C)$ is closed under finite colimits and limits.
\end{corollary}
\begin{proof}
Let 
$\phi\colon I\to
\tame(\pos, \C)$ be a functor 
indexed by a finite category $I$.
We have already proved  that 
its colimit in the category $\Fun(\pos, \C)$ is tame and coincides with  the colimit in the category
$\tame(\pos, \C)$ (see the discussion at the beginning of this section). 

For the limits, we use the same argument since, under the hypothesis that every tame functor in $\tame(\pos, \C)$
can be discretized along a functor admitting a transfer, 
the limit of $\phi$ in  $\Fun(\pos, \C)$
is also tame (see Proposition~\ref{prop_limtame}). This limit is therefore the limit
of $\phi$ in $\tame(\pos, \C)$.
\end{proof}

\begin{remark}\label{rem_sedwkjenow}
We observe that, in the particular case in which $\C$ is the category of finite-dimensional vector spaces, the category of finitely encodable functors with the additional assumptions of~\cite[Theorem 3.4]{waas} is abelian and, hence, closed under finite (co)limits. 
An example of a functor that satisfies those assumptions is the constant functor indexed by $\mathbb{R}$.
However, such a functor is not tame. 

Furthermore, even if tame functors with values in finite-dimensional vector spaces were a subcategory of the abelian category in~\cite[Theorem 3.4]{waas}, proving abelianity would still require showing closure under (co)limits because a subcategory of an abelian category may not be abelian.
\end{remark}

In the next section, we present some key examples of posets satisfying the assumption of Corollary~\ref{cor_closedcolim}.

\section{Posets of dimension at most 1}\label{sec:realizations}
To ensure that all tame functors indexed by a poset $\pos$  can be discretized by poset functors admitting transfers different hypotheses can be imposed on $\pos$. 
For example, being an upper semilattice, as investigated in~\cite{realisations_posets}.  
This section aims to introduce another class of examples of indexing posets that guarantee the existence of such discretizations. These are posets of dimension at most $1$.

\begin{definition}\label{sdgdfhdfgh}
Let $\pos$ be a poset. 
If no two different elements in $\pos$ are comparable, then we write $\dim(\pos)=0$ and call  $\pos$ a poset of \define{dimension $0$}.
If relations $w\leq u \leq x$ and $w\leq v \leq x$ in $\pos$ imply that $u$ and $v$
are comparable, then we write $\dim(\pos)\leq 1$ and call  $\pos$ a poset of \define{dimension at most $1$}.
If $0<\dim(\pos)\leq 1$ and $\pos$, then we say that $\pos$ is of \define{dimension $1$} and  write $\dim(\pos)=1$.
\end{definition}

\begin{figure}[h!]
\centering
\begin{subfigure}[h]{.26\linewidth}
\centering
\adjustbox{scale=.7}{
\begin{tikzcd}
\textcolor{white}{a} & & \\
	a_1 \arrow[r] & a_2 \arrow[r]& a_4 \\
	& a_3 \arrow[ur]& & 
\end{tikzcd}
}
\caption{}
\label{poset_dim1a}
\end{subfigure}
\begin{subfigure}[h]{.26\linewidth}
\centering
\adjustbox{scale=.7}{
\begin{tikzcd}
&[-2em] & & [-2em] b_4 \\
[-4ex]
& b_3 & & \\
[2ex]
b_1 \arrow[ur] \arrow[uurrr] & & b_2 \arrow[uur] \arrow[ul, crossing over] &
\end{tikzcd}
}
\caption{}
\label{poset_dim1b}
\end{subfigure}
\begin{subfigure}[h]{.29\linewidth}
\centering
\adjustbox{scale=.7}{
\begin{tikzcd}
& c_4 & \\
c_2 \arrow[ur]& & c_3 \arrow[ul] \\
& c_1 \arrow[ur] \arrow[ul] & &
\end{tikzcd}
}
\caption{}
\label{poset_dim2}
\end{subfigure}
\caption{Hasse diagram of three posets. 
The zig-zag poset (a) and the fence poset (b) are of dimension $1$ but  (c) is not.
}
\label{fig_ex_dim_poset}
\end{figure}

It is clear from this definition that a poset is of dimension at most $1$  if and only if for every two comparable elements there exists a unique monotone path between them. 
In particular,
if $\pos$ is a poset of dimension at most 1, then for every full subposet $\mathcal{S}$, $\dim (\mathcal{S})\le \dim (\pos)$.  
Examples of posets of dimension  $1$, in addition to the two depicted in Figure~\ref{fig_ex_dim_poset}, include $\mathbb{N}$,
$\mathbb{R}$, trees, and, in particular, zig-zags. 
The poset $\mathbb{R}^2$ and the posets in Figure~\ref{fig:gluing-posets} are examples of posets whose dimensions are not at most $1$.
\medskip 

For a poset $\pos$ and an element $x\in\pos$, the symbol $\mathcal{P}(x)$ denotes the set of all elements $y$ in $\pos$ covered by $x$, where an element $y$ in $\pos$ is covered by $x$ if $y<x$ and there is no $z$ such that $y<z<x$. 
\medskip

Let now $\pos$ be a poset of dimension at  most $1$, and $\R(\pos)$ be the disjoint union of $\pos$ and ${\mathcal{I}(\pos)}\coloneqq\coprod_{x\in \pos}\coprod_{y\in\mathcal{P}(x) } (-1,0)$,
where $(-1,0)$ is the interval of negative real numbers strictly bigger than $-1$.
Elements of $\mathcal{I}(\pos)$
are denoted as triples $(x,y,t)$ where $x$ covers $y$ in $\pos$ and $t$ is in $(-1,0)$.
Note that $\mathcal{I}(\pos)$ is non-empty
if and only if there is $x$ in $\pos$ for which
$\mathcal{P}(x)$ is non-empty. 
The symbols $\pi_0,\pi_{-1}\colon\R(\pos)\to\pos$ and $T\colon \R(\pos)\to (-1,0]$  denote the  functions:
\[\pi_0(z):=\begin{cases}
    z& \text{ if } z\in\pos\\
    x&\text{ if } z=(x,y,t)\in\mathcal{I}(\pos)
\end{cases}
\ \ \ \ \ \ \ \ \ \ 
\pi_{-1}(z):=\begin{cases}
    z& \text{ if } z\in\pos\\
    y&\text{ if } z=(x,y,t)\in \mathcal{I}(\pos)
\end{cases}\]
\[
T(z) \coloneqq 
\begin{cases}
    0& \text{ if } z\in\pos\\
    t&\text{ if } z=(x,y,t)\in\mathcal{I}(\pos)
\end{cases}
\]

We are going to consider the following relation $z\leq z'$ on the set $\R(\pos)$:
\begin{itemize}
\item if $z,z'$ are in $\pos$, then $z\leq z'$ in $\R(\pos)$
if and only if $z\leq z'$ in $\pos$;
\item if $z=(x,y,t)$  and 
$z'$ is in $\pos$, then
 $z\leq z'$ in $\R(\pos)$ if and only if
 $x\leq z'$ in $\pos$;
\item if $z$ is in  $\pos$ and $z'= (x',y',t')$, then 
$z\leq z'$ in $\R(\pos)$ if and only if
$z\leq y'$ in $\pos$;
\item if $z=(x,y,t)$ and $z'= (x',y',t')$, then $z\leq z'$ in $\R(\pos)$
if and only if either $x=x'$, $y=y'$, and
$t\leq t'$ or $x<x'$ and $y<y'$.
\end{itemize}

The introduced relation on $\R(\pos)$ is reflexive, antisymmetric, and transitive. 
The obtained poset $(\R(\pos), \leq)$ is called the \define{realization} of $\pos$.
This construction is a particular case of the more general realization of an arbitrary poset (not only of dimension at most $1$) introduced in~\cite{realisations_posets}. 
Note that the functions $\pi_0,\pi_{-1}$ are poset functors. Moreover, $z\leq z'$ in $\R(\pos)$ if and only if  
either $\pi_0(z)\leq \pi_{-1}(z')$, or $\pi_0(z)=\pi_0(z')$, $\pi_{-1}(z)=\pi_{-1}(z')$ and $T(z)\leq T(z')$.
In particular, for $x$ and $y$ in $\pos$, the relation $x\leq y$ holds in $\pos$ if and only if the analogous relation holds in
$\R(\pos)$. 

As an example of realization, consider the poset $\mathbb{N}$. 
Then $\R(\mathbb{N})$ is isomorphic to the poset of non-negative real numbers $[0,\infty)$. 
Indeed, the map $\R(\mathbb{N})\to [0,\infty)$ that sends 
$z$ to $\pi_0(z)+T(z)$ is an isomorphism of posets.
The realization construction has been originally introduced to mimic how the poset $[0,\infty)$ is obtained from $\mathbb{N}$.

\begin{proposition}\label{prop_real_dim1}
If $\pos$ is a poset of dimension at most $1$, then 
$\dim(\R(\pos)) = \dim(\pos)$.
\end{proposition}

\begin{proof}
If $\dim(\pos)=0$, 
then $\mathcal{I}(\pos)$ is empty and $\pos=\R(\pos)$. 
Let $\dim(\pos)=1$. 
To show $\dim(\R(\pos))=1$, 
consider the relations $w\leq u \leq z$ and $w\leq v \leq  z$ in $\R(\pos)$. 
We need to prove that $u$ and $v$ are comparable. 
By applying the poset functors  $\pi_0$ and $\pi_{-1}$ to the above relations and using the fact that $\dim(\pos)=1$, we obtain that $\pi_0(v)$ is comparable to  $\pi_0(u)$ and $\pi_{-1}(v)$ is comparable to $\pi_{-1}(u)$. 
Assume $\pi_0(v)\leq \pi_0(u)$. We now consider several cases:
\begin{itemize}
    \item   $u$ is in $\pos$. Then
    the relation  $\pi_0(v)\leq \pi_0(u)$ implies $v\leq u$.
    \item $u$ is in $\mathcal{I}(\pos)$  and $v$ is in $\pos$.
    If $\pi_0(v)= \pi_0(u)$, then $u\leq v$. 
    If $\pi_0(v)< \pi_0(u)$, then $\pi_{-1}(u)\geq \pi_{-1}(v)=\pi_0(v)$, and hence $u\geq v$.
    \item both $u$ and $v$ are in  $\mathcal{I}(\pos)$, i.e. $\pi_{0}(u)$ covers $\pi_{-1}(u)$ and $\pi_{0}(v)$ covers $\pi_{-1}(v)$.  In  this case the following relations have to hold:
    \[
\begin{tikzcd}[row sep=1pt, column sep = 5pt]
 & \pi_{-1}(u)\ar[draw=none, sloped ]{dr}[description]{\leq}\\
 \pi_{-1}(v)\ar[draw=none, sloped]{dr}[description]{\leq} \ar[draw=none, sloped]{ur}[description]{\leq} & & \pi_{0}(u)\ar[draw=none]{r}[description]{\leq} & \pi_0(z)\\
 & \pi_{0}(v)\ar[draw=none, sloped]{ur}[description]{\leq}
\end{tikzcd}
\]
Using again $\dim(\pos)=1$, we get 
that $\pi_{0}(v)$ and $\pi_{-1}(u)$ are comparable.
Consequently, either 
$\pi_{-1}(v)\leq \pi_{-1}(u)\leq \pi_{0}(v)\leq \pi_{0}(u)$
or $\pi_{-1}(v)\leq \pi_{0}(v)\leq \pi_{-1}(u)\leq \pi_{0}(u)$.
In the latter case, $v\le u$ in $\R(\pos)$.
In the former, there are different possibilities. 
If all the inequalities are equalities, then $u=v$ in $\pos$, but this is not the case by assumption. 
If either the first or the last inequality is strict, then $\pi_{-1}(u)=\pi_0(v)$, which implies that $v\le u$. 
If the middle inequality is strict, we have $\pi_{-1}(u)=\pi_{-1}(v)<\pi_0(u)=\pi_0(v)$, implying that $u$ and $v$ are comparable. \qedhere
\end{itemize}
\end{proof}

For a  subset $V\subset (-1,0)$, denote by $\R(\pos, V)\subset \R(\pos)$ the full subposet given by the disjoint union of $\pos$ and 
$\mathcal{I}(\pos,V)\coloneqq\coprod_{x\in \pos}\coprod_{y\in\mathcal{P}(x) } V$ (see Figure~\ref{fig_realisation}).
Thus, 
\begin{align*}
    \mathcal{I}(\pos,V) &= \{(x,y,t)\in \mathcal{I}(\pos)\ |\ t\in V\}, \\
    \R(\pos, V) &= \{z\in \R(\pos)\ |\ T(z)=0\text{ or } T(z)\in V\}.
\end{align*}
\begin{figure}[h!]
\centering
\adjustbox{scale=0.6}{
\begin{tikzcd}[column sep=2em,row sep=.8ex,ampersand replacement=\&]
\& \bullet \& \& \& [-2.4em] \& [-1.4em]\& [-2.4em]\bullet \& [-2.4em] \& [-1.4em]\& [-2.4em] \& \& \bullet \& \\
\& \& \& \& \& \bullet \arrow[ur,shorten >=-2pt,shorten <=-2pt] \& \& \bullet \arrow[ul,shorten >=-2pt,shorten <=-2pt] \& \& \& \& \& \\
\& \& \textcolor{white}{i}\arrow[r, hook] \& \textcolor{white}{i} \& \& \& \& \& \&\textcolor{white}{i} \arrow[r,hook] \& \textcolor{white}{i}\& \& \\
\& \& \& \& \bullet \arrow[uur] \& \& \& \& \bullet \arrow[uul] \& \& \& \& \\
\bullet \arrow[uuuur] \& \& \bullet \arrow[uuuul] \& \bullet \arrow[ur,shorten >=-2pt,shorten <=-2pt]\& \& \& \& \& \& \bullet \arrow[ul,shorten >=-2pt,shorten <=-2pt]\& \bullet  \arrow[uuuur,dash, thick, shorten >=-5pt,shorten <=-5pt] \& \& \bullet \arrow[uuuul,dash, thick,shorten <=-5pt,shorten >=-5pt]
\end{tikzcd}
}
\caption{From left to right: a poset $\pos$ of dimension 1, $\R(\pos, V)$ with $V=\{\frac{1}{4}, \frac{3}{4}\}$, and the realization $\R(\pos)$. In each poset, arrows point to greater elements.}
\label{fig_realisation}
\end{figure}

\begin{proposition}\label{wciojncwoiejwi}
Let $\pos$ be a poset of dimension at most 1. 
If $V\subset (-1,0)$ is finite, then the full inclusion $\alpha_V\colon \R(\pos, V)\subset \R(\pos)$ admits a transfer.
\end{proposition}
\begin{proof}
    Since, by Proposition~\ref{prop_real_dim1}, $\R(\pos)$ has dimension at most 1, $\R(\pos, V)$ has also dimension at most 1. 
Define  $\alpha_V^!\colon \R(\pos)\to \R(\pos, V)$ to be the function given by the formula:
\[
\alpha_V^!(z) = \begin{cases}
    z & \text{ if } z\in\pos\\
    y & \text{ if }  z=(x,y,t)\in\mathcal{I}(\pos)\text{ and } \{v\in V\ |\ v\leq t\}=\emptyset\\
    (x,y, \text{max}\{v\in V\ |\ v\leq t\})  & \text{ if } z=(x,y,t)\in\mathcal{I}(\pos)
    \text{ and } \{v\in V\ |\ v\leq t\}\not=\emptyset
\end{cases}
\]
Note that finiteness of $V$ is needed in the definition of  $\alpha_V^!$ since it involves taking \medskip maxima.

    The composition $\alpha_V^!\alpha_V\colon \R(\pos, V)\to \R(\pos, V)$ is the identity and, for any $z$ in $\R(\pos)$,
we have $\alpha_V \alpha_V^!(z)\leq z$.
This means  that $(\alpha_V^!)_{\ast}\colon \R(\pos)_{\ast}\to \R(\pos, V)_{\ast}$ is the transfer of $\alpha_V$ (see Proposition~\ref{ineq_transfer}).
\end{proof}

\begin{proposition}\label{afdfhsfg}
    Let $\C$ be a category closed under finite colimits and $\pos$ be a finite poset of dimension at most $1$. 
    Then, for every functor $X$ in  $\tame(\R(\pos), \C)$, there
    is a finite $V\subset (-1,0)$ for which the inclusion
    $\R(\pos, V)\subset \R(\pos)$ discretizes $X$. 
    In particular, all functors in $\tame(\R(\pos), \C)$ can be discretized by functors admitting a transfer. 
\end{proposition}

\begin{proof}
Let $X$ in $\tame(\R(\pos), \C)$ be discretized by the inclusion $\D\subset \R(\pos)$, with $\D$ finite. 
Let $V = T(\D)\setminus\{0\}$. 
Then $\D\subset \R(\pos, V)$. 
The assumption that $\pos$ is finite implies  $\R(\pos, V)$ is finite and consequently $X$ is discretized by 
$\alpha_V\colon \R(\pos, V)\subset \R(\pos)$ (see Proposition\ref{prop_full_lke}.1).  
Since $\alpha_V$ admits transfer (Proposition~\ref{wciojncwoiejwi}), the second statement also holds.
\end{proof}

\subsection{Realizations of some infinite posets of dimension at most 1}
In this section, we are going to extend Proposition~\ref{afdfhsfg} to certain infinite posets. 
\medskip 

A poset $\pos$ is of \define{finite type} if $\pos\leq x$ is finite for every $x$ in $\pos$.
For $U\subset \pos$, let us denote by $\text{\rm suplim}(U)$ the subset of $\pos$ consisting of all the elements in $\pos$ that bound from above $U$ and are minimal with that respect.
A subset $\D\subset \pos$ is called \define{closed} if $\text{\rm suplim}(U)\subset \D$  for every non-empty subset $U\subset \D$. 
The intersection of any family of closed subsets of $\pos$ is closed. 
Consequently, for any $\D\subset \pos$, there is a unique minimal closed subset $\overline{\D}\subset\pos$ (the intersection of all the closed subsets containing $\D$) such that 
$\D\subset \overline{\D}$. 
For posets of dimension at most $1$, elements of the set $\overline{\D}$ can be described explicitly:

\begin{proposition}
    If $\pos$ is a poset of dimension at most $1$, then, for every subset $\D\subset \pos$:
   \[
   \overline{\D}=\bigcup_{\substack{U\subset \D \\ U \text{ is  finite}}}\text{\rm suplim}(U)
   \]
\end{proposition}

\begin{proof}
    Let $D\coloneqq\bigcup_{\substack{U\subset \D \\ U \text{ is  finite}}}\text{\rm suplim}(U)$. 
    Clearly $\D\subset D\subset \overline{\D}$. 
    To show the proposition we need to prove therefore that $D$ is closed. 
    Let $T\subset D$ be a finite subset. 
    For each $t$ in $T$, let $U_t\subset D$ be such that $t$ is in
    $\text{\rm suplim}(U_t)$. 
    We claim $\text{\rm suplim}(T)\subset \text{\rm suplim}(\bigcup_{t\in T}U_t)$. 
    Assume to contrary that there is $x$ in $\text{\rm suplim}(T)\setminus \text{\rm suplim}(\bigcup_{t\in T}U_t)$.
    Since $y \leq x$ for every $y$ in $\bigcup_{t\in T}U_t$, there has to be $z$ such that $y\leq z<x$ for every $y$ in $\bigcup_{t\in T}U_t$. 
    If $y$ is in $U_t$, then the relations $y\leq t\leq x$ and the fact that $\pos$ is of dimension at most $1$ imply that $z$ is comparable to every $t$ in $T$. 
    The relation $z<t$ would contradict the fact that $t$ is in $\text{\rm suplim}(U_t)$. 
    We thus have $t\leq z$ for all $t$ in $T$. 
    This however contradicts the fact that $x$ is in $\text{\rm suplim}(T)$.
\end{proof}

This proposition gives:

\begin{corollary}\label{adfghfjdghkjfh}
     Let $\pos$ be a poset of dimension at most $1$. 
     Then the following statements are equivalent:
     \begin{enumerate}
         \item For all finite  $\D\subset\pos $, the set  $\overline{\D}$ is finite.
         \item For all finite  $\D\subset\pos $, the set $ \text{\rm suplim}(\D)$ is finite.
     \end{enumerate}
\end{corollary}

The reason we are interested in closed subsets is given by the following proposition.

\begin{proposition}\label{adgsgfjdgh}
    Let $\pos$ be a poset. 
    Assume that $\D\subset \pos$ is a closed full subposet satisfying the following condition: for every $x$ in $\pos$ for which $\D\leq x$ is non-empty, there is a unique $s$ in $\text{\rm suplim}(\D\leq x)$ such that $s\leq x$.
    Then $\D\subset \pos$ admits a transfer.
\end{proposition}

\begin{proof}
Define a function $\alpha^!\colon \pos_{\ast}\to \D_\ast$ as follows.
Set $\alpha^!(-\infty)=-\infty$.
Let $x$ be in $\pos$. 
If $\D\leq x$ is empty, define  $\alpha^!(x)\coloneqq-\infty$.
If $\D\leq x$ is non-empty, define  $\alpha^!(x)$ to be the unique element in $\text{\rm suplim}(\D\leq x)\subset \D$ for which
$\alpha^!(x)\leq x$. We claim that $\alpha^!$ is a poset functor. Let $x\leq y$ in $\pos$.
Assume $\D\leq x$ is non-empty. 
Since  $(\D\leq x)\subset (\D\leq y)$,  then
$\D\leq y $ is also non-empty. 
The relations $\alpha^!(x)\leq x\leq y$ imply 
$\alpha^!(x)$ belongs to $\D\leq y$, and consequently $\alpha^!(x)\leq\alpha^!(y)$.  
If $\D\leq x$ is empty, then  $\alpha^!(x)=-\infty \leq\alpha^!(y)$ also holds. 

Note that  $\alpha^!(x)=x$ for all $x$ in $\D$ and $\alpha^!(x)\leq x$ for all $x$ in $\pos$.
The functor $\alpha^!$ satisfies, therefore, the assumptions of Proposition~\ref{ineq_transfer} and hence it is the transfer of $\D\subset \pos$.
\end{proof}

\begin{corollary}\label{afdhfgsfdf}
Assume $\pos$ is a poset of dimension at most $1$ and of finite type. 
Then every closed full subposet inclusion $\D\subset \pos$  admits a transfer.
\end{corollary}

\begin{proof}
The proposition is clear if $\D$ is empty. 
Let $\D\subset\pos$ be a non-empty closed full subposet.
We are going to show that it satisfies the assumptions of Proposition~\ref{adgsgfjdgh}. 
Let $x$ in $\pos$ be such that $\D\leq x$ is non-empty. 
Recall that $\text{\rm suplim}(\D\leq x)$ consists of the minimal elements  $y$ in $\pos$  that bound from above $\D\leq x$.  
Since  $x$ bounds from above $\D\leq x$, and  $\pos\leq x$ is finite, there is $s$ in $\text{\rm suplim}(\D\leq x)$ for which $s\leq x$.
 
Assume that for $s$ and $s'$ in $\text{\rm suplim}(\D\leq x)$, both the relations $s\leq x$ and $s'\leq x$ hold.
This means that there is $d$ in $\D$ such that $d\leq s\leq x$ and $d\leq s'\leq x$. 
The dimension assumption then implies that $s$ and $s'$ are comparable. 
Minimality gives the equality $s=s'$. 
\end{proof}

\begin{corollary}\label{sdgfghjfgjkjkl}
Let $\C$ be a category closed under finite colimits.
Assume $\pos$ is a poset of dimension at most $1$ and of finite type such that, for every finite $\D\subset \pos$, the closed set $\overline{\D}$ is also finite. 
Then all functors in $\tame(\pos, \C)$ can be discretized by functors admitting a transfer. 
\end{corollary}

\begin{proof}
If $X$ in $\tame(\pos, \C)$ is discretized by $\D\subset \pos$, then it is also discretized by $\overline{\D}\subset \pos$, which, according to Corollary~\ref{afdhfgsfdf}, admits a transfer. 
\end{proof}

Let $\pos$ be a poset of dimension at most $1$. Choose $\D\subset \pos$ and a  subset  $V\subset (-1,0)$.
Define $\R(\pos, \D, V)\subset \R(\pos)$ to be the full subposet given by the disjoint union of $\D\subset \pos$
and $\mathcal{I}(\pos, \D, V)\coloneqq\coprod_{x\in \D}\coprod_{y\in\mathcal{P}(x) } V$. 
Thus: 
\begin{align*}
\mathcal{I}(\pos, \D, V) &= \{(x,y,t)\in \mathcal{I}(\pos)\ |\ x\in \D\text{ and }t\in V\},
\\
\R(\pos,\D, V) &= \{z\in \R(\pos)\ |\ \pi_0(z)\in \D \text{ and either }T(z)=0\text{ or } T(z)\in V\}.
\end{align*}
Note that since $\R(\pos)$ is of dimension at most 1 (see Proposition~\ref{prop_real_dim1}), then so is $\R(\pos, \D, V)$. 
Note also that $\R(\pos,\D, V)$ is a full subposet of $\R(\pos, V)$.

\begin{proposition}\label{sfgdgjgfjkhjfm}
    Let $\pos$ be a poset of dimension at most $1$  and $V\subset (-1,0)$ a subset. 
    If $\D\subset \pos$ is closed, then $\R(\pos,\D, V)$ is closed in $\R(\pos, V)$.
\end{proposition}

\begin{proof}
Consider a subset $U\subset \R(\pos,\D, V)$ and a minimal element $s$ in $\R(\pos, V)$ that bounds from above $U$. 
It follows that $\pi_0(s)$ bounds from above $\pi_0(S)$ in $\pos$. 
Moreover, $\pi_0(s)$ has to be a minimal such bounding element. 
Thus $\pi_0(s)$ belongs to $\D$, since $\D$ is closed, and consequently $s$ is in $\R(\pos,\D, V)$.
\end{proof}

\begin{corollary}
Assume $\pos$ is a poset of dimension at most $1$ and of finite type, $\D\subset \pos$ is closed, and $V\subset (-1,0)$ is a finite subset. 
Then the inclusion $\R(\pos,\D, V)\subset \R(\pos)$ admits a transfer.
\end{corollary}

\begin{proof}
By Proposition~\ref{prop_real_dim1}, the poset $\R(\pos)$ has dimension at most 1. 
Since  $\R(\pos, V)$  is a full subposet of $\R(\pos)$, it is of 
finite type and dimension at most $1$.  
We can then use Proposition~\ref{sfgdgjgfjkhjfm} and Corollary~\ref{afdhfgsfdf} to conclude that $\R(\pos,\D, V)\subset \R(\pos, V)$ admits a transfer. 
Since $\R(\pos, V)\subset \R(\pos)$ also admits a transfer (here we use the fact that $V$ is finite, see the discussion before Proposition~\ref{afdfhsfg},
so does the composition $\R(\pos,\D, V)\subset \R(\pos)$.
\end{proof}

We are now ready to generalize Proposition~\ref{afdfhsfg}:

\begin{theorem}\label{adfgsdfjdghn}
    Let $\C$ be a category closed under finite colimits and $\pos$ be a finite type poset of dimension at most $1$ such that, for every finite $\D\subset \pos$, the set $ \text{\rm suplim}(\D)$ is also finite.
    Then, for every functor $X$ in  $\tame(\R(\pos), \C)$, there is a finite $V\subset (-1,0)$ and a finite closed $\D\subset \pos$ for which the inclusion $\R(\pos,\D, V)\subset \R(\pos)$ discretizes $X$. 
    In particular, all functors in $\tame(\R(\pos), \C)$ can be discretized by functors admitting a transfer. 
\end{theorem}

\begin{proof}
Let $X$ in $\tame(\R(\pos), \C)$ be discretized by a finite ${{\mathcal U}}\subset \R(\pos)$.
Set $\D\coloneqq\overline{\pi_0({\mathcal U})}$ and $V\coloneqq T({\mathcal U})\setminus\{0\}$.
Then $\D\subset \pos$ is a finite (see Corollary~\ref{adfghfjdghkjfh}) and closed  subset, $V$ is  finite, and ${\mathcal U}\subset \R(\pos,\D, V)$. 
Since $\R(\pos,\D, V)$ is finite, it also discretizes $X$.
\end{proof}

\section{Abelian structure on tame functors}\label{sec:tame_ab}
The goal of this section is to show:

\begin{theorem}\label{th_tame_ab}
Let $\C$ be an abelian category and $\pos$  a poset. Assume that every functor in $\tame(\pos, \C)$ can be discretized by a poset functor admitting a transfer.  Then:
\begin{enumerate}
    \item  The category $\tame(\pos, \C)$ is abelian.
    \item If $\C$ has enough projectives, then so does $\tame(\pos, \C)$.
    \item If  $\C$ satisfies the minimal cover axiom, then so does $\tame(\pos, \C)$.
    \item If $\C$ has enough projectives and $\pos$ is of dimension at most $1$,
    then:
    \[\text{\rm proj-dim}(\tame(\pos, \C))\leq \text{\rm proj-dim}(\C)+1.\]
\end{enumerate}
\end{theorem}

It is known that if $\D$ is a small category, then 
the category of functors $\Fun\left(\D,\C\right)$ inherits the abelian structure from $\C$ (see~\cite[Prop.1.4.4]{borceux2}). 
Thus most of the work needed to prove the first statement of Theorem~\ref{th_tame_ab} has been done in Section~\ref{S_transfer}, by showing the existence of finite (co)limits in $\tame(\pos,\C)$ under the assumption that its functors can be discretized along poset functors admitting transfers. 

\begin{proof}[Proof of Theorem~\ref{th_tame_ab}.1]
The category $\Fun(\pos, \C)$ is abelian since $\pos$ is small and $\C$ is abelian. 
As $\tame(\pos, \C)$ is a full subcategory of $\Fun(\pos, \C)$, its morphism sets are abelian groups and the composition of morphisms is bilinear. 
Moreover,
by Corollary~\ref{cor_closedcolim}, the category 
$\tame(\pos, \C)$ is closed under finite limits and colimits, and the inclusion $\tame(\pos, \C)\subset \Fun(\pos, \C)$ commutes with these operations (see Remark~\ref{deocnkwopewc}). 
In particular, the $0$ object, kernels, and cokernels of morphisms in $\tame(\pos, \C)$ exist and they coincide with the analogous objects in $\Fun(\pos, \C)$. 
Consequently, the category  $\tame(\pos, \C)$ is abelian.
\end{proof}

The last three statements of Theorem~\ref{th_tame_ab} are going to be proven in two steps. 
The first step is to reduce to the case of $\pos$ being finite.

\begin{proof}[Proofs of Theorem~\ref{th_tame_ab}.2-3  assuming they are true for finite posets]
Consider a finite full subposet inclusion $\alpha\colon \D\subset \pos$  that   discretizes a tame functor $X\colon\pos\to \C$ 
(see Proposition~\ref{prop_full_lke}.1). 
Recall that the natural transformation $\alpha^k (X\alpha)\to X$ (adjoint to the identity) is an isomorphism (see Proposition~\ref{prop_full_lke}.4).
Let $p\colon P_0\to X\alpha$ be a (minimal) projective cover in $\Fun(\D,\C)$, which exists by our assumption that statements $2$ and $3$ of the theorem are true for finite posets.
Since left Kan extensions preserve the properties of being an epimorphism and of being projective, the morphism $\alpha^k p\colon \alpha^k P_0\to \alpha^k (X\alpha)\cong X$ is a projective cover of $X$ in $\tame(\pos,\C)$. That proves statement~2.

To show statement~3, we are going to prove that  $\alpha^k p$ is minimal if $p$ is minimal. 
Assume $p$ is minimal and consider an endomorphism $\varphi\colon \alpha^k P_0\to \alpha^k P_0$ such that $(\alpha^k p) \varphi=\alpha^k p$. 
Precomposing with $\alpha$ gives the following diagram
\[
\begin{tikzcd}
P_0\arrow[rr, "\cong"]\arrow[d, swap, "\cong"] & & P_0\arrow[d, "\cong"]
\\
(\alpha^k P_0)\alpha\arrow[d, two heads]\arrow[dr, two heads]\arrow[rr, "\psi=\varphi^{\alpha}"] & & (\alpha^k P_0)\alpha\arrow[d, two heads]\arrow[dl, two heads]
\\
(\alpha^k (X\alpha)) \alpha\arrow[r, pos=.3,"\cong"] & X\alpha &(\alpha^k (X\alpha)) \alpha\arrow[l, pos=.3,"\cong"']
 \end{tikzcd}
\]
where the isomorphisms from bottom to top are given, respectively, by Proposition~\ref{prop_full_lke}.4, Proposition~\ref{prop_full_lke}.3, and the minimality of $p\colon P_0\to X\alpha$.
This implies that $\psi$ is also an isomorphism and, thus, its left Kan extension $\alpha^k\psi=\alpha^k(\varphi^\alpha)$ is an isomorphism.
We can conclude that $\varphi$ is an isomorphism by Proposition~\ref{prop_full_lke}.6.
\end{proof}

\begin{proof}[Proof of Theorem~\ref{th_tame_ab}.4 under the assumption that it is true for finite posets]
Consider a finite full subposet inclusion $\alpha\colon \D\subset \pos$  that admits transfer and  discretizes a tame functor $X\colon\pos\to \C$ 
(see Corollary~\ref{dsgdfh}.3). 
Let $P\to X\alpha$ be a projective resolution in $\Fun(\D,\C)$ of length smaller or equal to $\text{\rm proj-dim}(\C)+1$ which exists since $\D$ has dimension at most $1$ and we are assuming that the statement is true for finite posets. 
Recall that left Kan extensions in general preserve the property of being projective. 
Moreover, in this case, the functor $\alpha^k$ is also exact, because precomposing with any functor, in particular a transfer, is exact. 
Consequently, $\alpha^k(P)\to \alpha^k(X\alpha)\cong X$ is a projective resolution of $X$ in 
$\tame(\pos, \C)$ of length smaller or equal to 
$\text{\rm proj-dim}(\C)+1$.
\end{proof}

It remains to prove the last three statements of Theorem~\ref{th_tame_ab} for finite posets. 
The strategy is to use the following standard construction.
Let $\D$ be a finite poset, $x$ an element in $\D$, and $0\leq 1$ be the poset with two elements and the indicated relation.
Define ${\mathcal K}^x\colon \Fun(\D,\C)\to \Fun(0\leq 1,\C)$ to be the functor that assigns to a natural transformation $f\colon Y\to X$ in $\Fun(\D,\C)$ the following morphism in  $\Fun(0\leq 1,\C)$:

\[
\begin{tikzcd}[column sep=5em, row sep=2em]
{\mathcal K}^x(Y)\ar[swap]{d}{{\mathcal K}^x(f)}\ar[phantom, "\coloneqq"]{r}  &
\displaystyle \bigoplus_{y\in \mathcal{P}(x)} Y(y)\ar{rr}{  \Sigma_{y\in \mathcal{P}(x)}Y(y\leq x)}
\ar[d, "\bigoplus_{y\in \mathcal{P}(x)}f(y)", swap]
& & \displaystyle  Y(x)\ar{d}{f(x)}
\\
{\mathcal K}^x(X)\ar[phantom, "\coloneqq"]{r}  &
\displaystyle \bigoplus_{y\in \mathcal{P}(x)} X(y)\ar{rr}{  \Sigma_{y\in \mathcal{P}(x)}X(y\leq x)} & & \displaystyle  X(x)
\end{tikzcd}
\]
By composing ${\mathcal K}^x$ with the (co)kernel of the transition morphism in $\mathcal{K}^x(-)$,
we obtain two functors which we denote by $H_0^x$ and $H_1^x$ :
\[
\begin{tikzcd}
 & & \C
 \\
\Fun(\D,\C)\ar{r}{\mathcal{K}^{x}}\ar[bend left=20pt]{urr}{H_1^x} 
\ar[swap, bend right=20pt]{drr}{H_0^x} & \Fun(0\leq 1, \C) \ar[swap]{ru}{\ker} \ar{rd}{\coker}\\
& & \C
\end{tikzcd}
\]

For example, for an object $A$ in $\C$, and an element $z$ in $\D$,  consider the homogenous  free functor $A(z,-)\colon \D\to\C$ (see the beginning of Section~\ref{sec_tameness}). 
Then:
\[
H_0^x(A(z,-)) =
\begin{cases}
    A & \text{if } z=x\\
    0  & \text{if } z\not=x
\end{cases}
\ \ \ \ \ \ \ \ \ 
H_1^x(A(z,-))=\begin{cases}
A^{|\{y\in {\mathcal P}(x)\, |\, y\geq z\} |-1} &\text{if } z< x\\
0 & \text{otherwise }
\end{cases}
\]

This in particular implies:

\begin{proposition}\label{sdgdfhfsj}
Let $A$ be an object in an abelian category $\C$. 
Assume $\D$ is a finite poset of dimension at most $1$. 
Then $H_1^x(A(z,-))=0$ for all elements $x$ and  $z$ in $\D$.
\end{proposition}

The functors $\mathcal{K}^x$, $H_0^x$, and $H_1^x$ enjoy the following  properties:
\begin{proposition}\label{sfsfdhfgs}
    Let $\C$ be an abelian category and $\D$ a finite poset. Then
\begin{enumerate}
    \item For every $x$ in $\D$, the functor ${\mathcal K}^x\colon \Fun(\D,\C)\to \Fun(0\leq 1,\C)$  is  exact.
    \item  For every $x$ in $\D$, a short  exact sequence $0\to Z\to Y\to X\to 0$ in $\Fun(\D,\C)$
    leads to an exact sequence in $\C$ of the form:
    \[\begin{tikzcd}
 0\rar &     H_1^x(Z) \rar & H_1^x(Y) \rar &H_1^x(X) \ar[out=-30, in=150]{dll}
 \\      
&  H_0^x(Z) \rar & H_0^x(Y) \rar &H_0^x(X) \rar & 0
    \end{tikzcd}\]
\item A natural transformation  $f\colon Y\to X$ is an epimorphism in $\Fun(\D,\C)$ if and only if for every $x$ in $\D$,
$H_0^x(f) \colon H_0^x(Y)\to H_0^x(X)$ is an epimorphism in $\C$ .
\item $X\colon\D\to\C$ is the zero functor if and only if $H_0^x(X)=0$ for every $x$ in $\D$.
\item For every $x$ in $\D$, the functor $H_0^x\colon\Fun(\D,\C)\to \C$ is a left adjoint.
\end{enumerate}
\end{proposition}
\begin{proof}
\noindent \textit{1}:\quad It is a consequence of the direct sum being exact.
\medskip

\noindent \textit{2}:\quad It follows from the exactness (statement 1) and the standard Snake Lemma.
\medskip

\noindent \textit{3}:\quad  If $f$ is an epimorphism, then, for every $x$ in $\D$, so is ${\mathcal K}^x(f)$ and consequently so is $H_0^x(f)$. Assume $H_0^x(f)$ is an epimorphism for every $x$ in $\D$. 
By contradiction suppose $f$ is not an epimorphism. Since $\D$ is finite, we can then choose 
a minimal element $z$ in $\D$ for which $f(z)\colon Y(z)\to X(z)$ is not an epimorphism. Consider the following commutative diagram with exact rows: 
\[
\begin{tikzcd}
\bigoplus_{y\in {\mathcal P}(z)} Y(y)\arrow[r]\arrow[d] & Y(z)\arrow[r] \arrow[d, "f(z)"] & H_0^z(Y)\arrow[d, "H_0^z(f)"]\ar{r} & 0\\
\bigoplus_{y\in {\mathcal P}(z)} X(y)\arrow[r] & X(z)\arrow[r]  & H_0^z(X)\ar{r} & 0
\end{tikzcd}
\]
The hypothesis implies that the right vertical morphism is an epimorphism. 
The left vertical morphism is also an epimorphism, by minimality of $z$ among the elements of $\D$ for which $f(z)$ is not an epimorphism.
Thus, the central vertical morphism is also an epimorphism, which is a contradiction. 
\medskip

\noindent \textit{4}:\quad  It is a consequence of statement 3
by considering the morphism $0\to X$.
\medskip

\noindent \textit{5}:\quad  The right adjoint is given by the functor 
${\mathcal U}_x\colon \C\to\Fun(\D,\C)$ which
\begin{itemize}
    \item to an object $A$ in $\C$ assigns  the unique functor such that:
\[
{\mathcal U}_x(A)(z)=\begin{cases}
    A &\text{ if } z=x\\
    0 &\text{ if } z\not=x
\end{cases}
\]
\item to  $f\colon A\to B$ in $\C$  assigns  the unique natural transformation ${\mathcal U}_x(f)\colon {\mathcal U}_x(A)\to {\mathcal U}_x(B)$ for which $({\mathcal U}_x(f))(x)=f$.\qedhere
\end{itemize}
\end{proof}

Proposition~\ref{sfsfdhfgs} can be used to describe how to construct (minimal) projective covers in $\Fun(\D,\C)$ when $\D$ is finite.

Let $X\colon\D\to\C$ be a functor. Assume that, for every $z$
in $\D$, we are given a (minimal) projective cover $c_z\colon P_z\to H_0^z(X)$ in $\C$.
For every $z$ in $\D$, let us choose a morphism $s_z\colon P_z\to X(z)$ whose composition with the quotient morphism $X(z)\to H_0^z(X)$ is the cover $c_z$. Such a lift exists since $P_z$ is projective.
This morphism determines a natural transformation from the homogeneous free functor $P_z(z,-)\to X$ (see Section~\ref{sec_tameness}).  
Let $s\colon \bigoplus_{z\in\D}P_z(z,-)\to X$ be their sum. 
\begin{lemma}\label{daDGSFJKJ}
    The natural  transformation $s\colon \bigoplus_{z\in\D}P_z(z,-)\to X$ is a  (minimal) projective cover of $X$ in $\Fun(\D,\C)$.
 \end{lemma}
 \begin{proof}
Since left Kan extensions and direct sums preserve the property of being projective,  the functor $\bigoplus_{z\in\D} P_z(z,-)$ is projective in $\Fun(\D,\C)$. 
Note that $H_0^x(\bigoplus_{z\in\D} P_z(z,-))= P_x$  
and $H_0^x(s)$ coincides with $s_x$ for every $x$ in $\D$. 
Thus according to Proposition~\ref{sfsfdhfgs}.3, $s$ is an epimorphism. 
The natural transformation $s$ is therefore a projective cover of $X$. 

To show the minimality of $s$, assume that $s_x$ is a lift of a minimal projective cover $c_x\colon P_x\to H_0^x(X)$ in $\C$. 
Let $\varphi\colon \bigoplus_{z\in\D} P_z(z,-)\to \bigoplus_{z\in\D} P_z(z,-)$ be a natural transformation for which $s=s\varphi$. 
Since, for every $x$, the morphism $H_0^x(s)=c_x$ is  minimal  in $\C$, the morphism $H_0^x(\varphi)$
is an isomorphism. 
Consequently $\varphi$ is an epimorphism (see Proposition~\ref{sfsfdhfgs}.3). 
Let $\psi\colon \bigoplus_{z\in\D} P_z(z,-)\to \bigoplus_{z\in\D} P_z(z,-)$ be such that $\varphi\psi$ is the identity, which exists since $\bigoplus_{z\in\D} P_z(z,-)$ is projective.  
The natural transformation $\psi$ is therefore a monomorphism. 
Moreover, $s\psi=s\varphi\psi= s$. 
By the previous argument, $\psi$ is also an epimorphism, and hence an isomorphism.
\end{proof}

Theorem~\ref{th_tame_ab}.2 and Theorem~\ref{th_tame_ab}.3 for finite posets are direct consequences of Lemma~\ref{daDGSFJKJ}. 
This lemma can also be used to characterize
projective objects in $\Fun(\D,\C)$ for finite $\D$:

\begin{proposition}\label{afashdsfjg}
Let $\C$ be an abelian category and  $\D$ be a finite poset. 
\begin{enumerate} 
\item  A functor $X\colon \D\to\C$  is  projective in $\Fun(\D,\C)$
if and only if
$H_0^z(X)$ is projective in $\C$,  for every $z$ in $\D$, and 
$X$ is isomorphic to  
 $\bigoplus_{z\in\D} H_0^z(X)(z,-)$.
\item If, in addition, $\D$ is a poset of dimension at most $1$, then a functor 
$X\colon \D\to\C$ is projective if and only if  $H_1^x(X)=0$ and $H_0^x(X)$
is  projective  in $\C$ for every $x$ in $\D$.
\end{enumerate}
\end{proposition}
\begin{proof}
\noindent \textit{1}:\quad 
Since left Kan extensions and direct sums preserve the property of being projective, if
$H_0^z(X)$ is projective in $\C$ for every $z$ in $\D$, then  $\bigoplus_{z\in\D} H_0^z(X)(z,-)$ is projective in
$\Fun(\D,\C)$. 
This shows the sufficient condition. 

Assume $X$ is projective. 
Since the functor $H_0^x$ is a left adjoint (see Proposition~\ref{sfsfdhfgs}.5), the object $H_0^x(X)$ is projective in $\C$ for every $x$ in $\D$. Thus the minimal projective cover given by Lemma~\ref{daDGSFJKJ} gives the desired isomorphism between $X$ \medskip and 
$\bigoplus_{x\in\D} H_0^x(X)(x,-)$.

\noindent \textit{2}:\quad  The necessary condition follows from statement~1
and Proposition~\ref{sdgdfhfsj}. 
To show the sufficient condition, consider the minimal 
projective cover $s\colon \bigoplus_{x\in\D}H_0^x(X)(x,-)\to X$
given by Lemma~\ref{daDGSFJKJ}. 
We need to show that its kernel $Z\coloneqq\ker(s)$ is trivial. 
The assumption $H_1^x(X)=0$ and the exact sequence given in Proposition~\ref{sfsfdhfgs}.2  gives $H_0^x(Z)=0$. 
Thus by Proposition~\ref{sfsfdhfgs}.4 we can conclude that $Z=0$.
\end{proof}

We can now characterize projective objects in $\tame(\pos,\C)$. 
They coincide with finite sums of homogeneous free objects, analogously to the case of finitely generated projective modules over graded rings, PIDs, or local rings. 

\begin{corollary}\label{asfddhdhfj}
Let $\C$ be an abelian category and $\pos$ a poset. 
Assume that every functor in $\tame(\pos, \C)$ can be discretized by a poset functor admitting a transfer. 
Then $X$ in  $\tame(\pos,\C)$ is projective if and only if there is a finite $\D\subset\pos$  and a sequence $\{P_z\}_{z\in\D}$ of projective objects in $\C$ for which
$X$ is isomorphic to $\bigoplus_{z\in\D} P_z(z,-)$.
\end{corollary}

We finish this section with:

\begin{proof}[Proof of Theorem~\ref{th_tame_ab}.4  for finite posets]
Assume $\pos$ is finite. 
If $\text{\rm proj-dim}(\C)=\infty$, the statement is true.
Assume then $\text{\rm proj-dim}(\C)$ is finite. 
Consider a minimal projective cover of a functor $X\colon\pos\to\C$ discussed in Lemma~\ref{daDGSFJKJ}, $s\colon \bigoplus_{z\in\pos}P_z(z,-)\to X$. 
Note that, for every $x$ in $\pos$, we have $H_1^x(\ker(s))=0$ as a consequence of Proposition~\ref{sdgdfhfsj} and the exact sequence given in Proposition~\ref{sfsfdhfgs}.2. 
The statement would then be proven if we showed that every functor $Y\colon \pos\to\C$, such that $H_1^x(Y)=0$ for every $x$ in $\pos$, has a projective resolution of length at most $m_Y\coloneqq\max\{\text{\rm proj-dim}(H_0^x(Y))\, |\, x\in\pos\}$. 
We do that by induction on $m_Y$. 

If $m_Y=0$, then $H_0^x(Y)$ is projective for every $x$ and hence, according to Proposition~\ref{afashdsfjg}.2, $Y$ is projective. 
The claim is therefore true in this case. 
Let $m_Y>0$ and assume that the claim is true for all functors $X\colon\pos\to\C$ such that $m_X<m_Y$ and $H_1^x(X)=0$ for all $x$ in $\pos$. 
Consider a projective cover $s\colon \bigoplus_{z\in\pos}P_z(z,-)\to Y$ discussed in Lemma~\ref{daDGSFJKJ}. 
For every $x$ in $\pos$, there is an exact sequence (see Proposition~\ref{sfsfdhfgs}.2):
\[
0\to H_0^x(\ker(s))\to
H_0^x\left(\bigoplus_{z\in\pos}P_z(z,-)\right)=P_x\to 
H_0^x(Y)\to 0.
\]
This means that $m_{\ker(s)}< m_Y$. 
Since also $H_1^x(\ker(s))=0$ for every $x$, by the inductive assumption the functor $\ker(s)$ has a projective resolution of length at most $m_Y-1$. 
By concatenating this resolution with $s$ we obtained a resolution of $Y$ of length at most $m_Y$. 
\end{proof}

\section{Model structure on tame functors}
\label{subsec:tame_mod}
In this section, we explain how to extend a model structure on $\C$ to a model structure on $\tame(\pos,\C)$ provided that every functor in $\tame(\pos,\C)$ can be discretized along a poset functor admitting a transfer (see Theorem~\ref{th_tame_mod}). 
This is a generalization of the known result of \cite[Prop.10.11]{dwyerspalinski} for $\Fun(\D, \C)$ with $\D$ finite, which we extend in Theorem~\ref{def_finite_cofibration} to include minimality.
It is also a new result with respect to \cite[Th.9.3]{realisations_posets} and \cite[Th.2.2]{bcw}, where the indexing poset $\pos$ is required to be an upper semilattice and $[0,\infty)$, respectively.

\begin{theorem}\label{def_finite_cofibration}
Let $\D$ be a finite poset and $\C$ a model category. 
Then the following choices of weak equivalences, fibrations, and cofibrations 
in  $\Fun\left(\D,\C\right)$ form a model structure.
A morphism $\varphi\colon X\to Y$ in $\Fun\left(\D,\C\right)$ 
is a
\begin{itemize}
    \item \define{weak equivalence}/\define{fibration} if,
    for every $x$ in $\D$, the morphism
    $\varphi(x)\colon X(x)\to Y(x)$ is a weak equivalence/fibration in $\C$;
    \item \define{cofibration}  if, for every $x\in \D$, the mediating morphism
\begin{equation}\label{mediating_morph_cofibration}
\widehat{\varphi}(x)\colon P=\colim\left(\colim_{\D<x} Y \leftarrow \colim_{\D < x} X \rightarrow X(x)\right) \to Y(x)
\end{equation}
in the pushout diagram 
\begin{equation*}\label{D_finite_cofib}
\begin{tikzcd}[column sep=2em, row sep=3ex, ampersand replacement=\&]
		\colim_{\D < x} X \ar[r] \ar[d, "\colim_{\D < x} \varphi"'] \arrow[dr, phantom, "\scalebox{1.5}{\color{black}$\ulcorner$}", pos=.9]
		\& X(x) \ar[d] \ar[ddr, bend left=35]
		\& [-1em]
		\\
		\colim_{\D < x} Y \ar[r] \ar[drr, bend right=16]
		\& P\ar[dr, "\widehat{\varphi}(x)", pos=0.3]
		\&
		\\
		[-1ex]
		\&
		\& Y(x)
\end{tikzcd}
\end{equation*}
is a cofibration in $\C$.
\end{itemize}

Furthermore, if $\C$ satisfies the minimal cofibrant factorization axiom, then so does $\Fun\left(\D,\C\right)$.
\end{theorem}

\begin{proof}
For the proof of the model structure part see~\cite{dwyerspalinski}. 

To construct minimal cofibrant factorizations our strategy is the same as the one used in~\cite{dwyerspalinski} to construct cofibrant factorizations, except that at each step we take minimal cofibrant factorizations in $\C$. 
Explicitly, consider $\varphi\colon X\to Y$ in $\Fun\left(\D,\C\right)$.
Let $\mathcal{M}\subset \D$ be the set of minima of $\D$.
We claim that a minimal cofibrant factorization of $\varphi$ is constructed inductively by performing the following steps: 
\begin{itemize}
    \item[(i)] For all $m\in \mathcal{M}$, take a minimal cofibrant factorization of $\varphi(m)$;
    \item[(ii)] For all $x\in \D\setminus\mathcal{M}$, first, take a minimal cofibrant factorization $\colim_{\D<x}X\to W\to \colim_{\D<x}Y$ and then take a minimal cofibrant factorization $P\to C(x)\to Y(x)$ of the mediating morphism of the pushout $P=\colim(X(x)\leftarrow\colim_{\D<x}X\to W)$:
\[
\begin{tikzcd}[ampersand replacement=\&]
    \colim_{\D < x} X \arrow[r]\ar{d} \arrow[dr, phantom, "\scalebox{1.5}{\color{black}$\ulcorner$}", pos=.9] 
    \& W\arrow{rr}\ar{d} 
    \&  
    \&  \colim_{\D < x} Y\ar{d}
    \\
    X(x) \ar{r}\ar[bend right=25pt, "\varphi(x)"]{rrr} \& P\ar{r} \& C(x)\ar{r} \& Y(x)
\end{tikzcd}
\]
The wanted factorization consists, for every $x$, of the cofibration given by the composition $X(x)\to P \to C(x)$ and the fibration and weak equivalence $ C(x)\to Y(x)$.
\end{itemize}
We already know (see~\cite{dwyerspalinski}) that this construction gives a factorization into a cofibration and a fibration that is also a weak equivalence. 
The minimality of this construction follows by the universal property of the pushout (see~\cite[Prop. 5.14]{barbara_thesis} for more details). 
\end{proof}

Since we define fibrations, rather than cofibrations, objectwise, the above-defined model structure is usually referred to as a \define{projective model structure}.
\medskip

The goal of this section is to extend Theorem~\ref{def_finite_cofibration} to: 

\begin{theorem}\label{th_tame_mod}
Let $\C$ be a model category. 
Assume that each functor in $\tame(\pos, \C)$ can be discretized along a poset functor admitting a transfer.
Then the following choices of morphisms in $\tame(\pos, \C)$ define a 
model category structure. 
A morphism  $\varphi\colon X\to Y$ in  $\tame\left(\pos,\C \right)$
is:
\begin{itemize}
    \item A \define{weak equivalence/fibration} if, for every $x$ in $\pos$,  the morphism $\varphi(x)\colon X(x)\to Y(x)$ is a weak equivalence/fibration  in $\C$;
    \item  A \define{cofibration} if  there is a finite full subposet $\alpha\colon \D\subset \pos$ discretizing both $X$ and $Y$ and for which  
    the restriction $\varphi^\alpha\colon X\alpha\to Y\alpha$ is a cofibration in $\Fun\left(\D,\C\right)$ as described in Thereom~\ref{def_finite_cofibration}.
\end{itemize} 

Furthermore, if $\C$ satisfies the minimal cofibrant factorization axiom, then so does $\tame\left(\pos,\C\right)$.
\end{theorem}

\begin{proof}
Property (MC1) of a model category is satisfied by Corollary~\ref{cor_closedcolim}.   
Properties (MC2) and (MC3) are clear from the definitions.
To show (MC4), let $\varphi$ be a cofibration in $\tame\left(\pos,\C\right)$. 
Let $\alpha\colon \D\subset \pos$ be a finite full subposet discretizing both $X$ and $Y$ and for which the restriction $\varphi^\alpha\colon X\alpha\to Y\alpha$ is a cofibration in $\Fun\left(\D,\C\right)$. 
Consider the following commutative diagrams:
\[
\begin{tikzcd}
\alpha^k(X\alpha)\ar{r}{\cong}\ar{d}[swap]{\alpha^k(\varphi^\alpha)} & X\ar{d}[swap]{\varphi} \ar{r} & E\ar{d}{\psi}\\
\alpha^k(Y\alpha)\ar{r}{\cong} & Y\ar{r}\ar[dotted]{ur} & B
\end{tikzcd}
\ \ \ \ \ \ \ \ \ \ 
\begin{tikzcd}
X\alpha\ar[swap]{d}{\varphi^\alpha}\ar{r} & E\alpha\ar{d}{\psi^{\alpha}}\\
Y\alpha\ar{r}\ar[dotted]{ur}{s} & B\alpha
\end{tikzcd}
\]
where
\begin{itemize}
    \item $\psi$ is a fibration and either $\varphi$ or $\psi$ is a weak equivalence;
    \item the solid horizontal arrows in the right diagram correspond via adjunction to the  morphisms represented by the compositions of the solid horizontal arrows in the left diagram;
    \item $s$ is a lift given by the (MC4) axiom
    for the model structure on  $\Fun\left(\D,\C\right)$ described in Theorem~\ref{def_finite_cofibration};
    \item the dotted arrow in the left diagram corresponds via adjunction to $s$.
\end{itemize}
A desired lift is then given by the dotted arrow in the left diagram. 
Note that, in the argument showing the lifting axiom (MC4), the fact that all tame functors can be discretized along a poset functor admitting a transfer is not used. 
Transfer is needed to ensure the factorization axiom (MC5) and the minimality statement. 

While precomposing preserves objectwise weak equivalences and fibrations, taking the left Kan extension in general does not. 
Nevertheless, if $ \alpha\colon \D\to \pos$ is a poset functor admitting a transfer, then the left Kan extension along $\alpha$ 
also preserves objectwise weak equivalences and fibrations because it can be expressed as a precomposition with the transfer.
To show (MC5) and the minimality statement,  consider a  morphism $\varphi\colon X\to Y$ in $\tame\left(\pos,\C\right)$.  Since tame functors are closed under colimits, the coproduct
$X\coprod Y$ is tame. We can then choose a finite full subposet $\alpha\colon \D\subset \pos$ that admits a transfer and discretizes this coproduct, and hence discretizes both $X$ and $Y$. Consider  a factorization
$X\alpha\to Z\to Y\alpha$ of $\varphi^\alpha$ given by the model structure 
on  $\Fun\left(\D,\C\right)$ (see Thereom~\ref{def_finite_cofibration}).
Then a desired factorization of $\varphi$ is given by the left Kan extensions $\alpha^k(X\alpha\to Z \to Y\alpha)$. 
This shows (MC5). 
If the initial factorization is 
a minimal cofibrant factorization of $\varphi^{\alpha}$, then, according to Proposition~\ref{prop_full_lke}.6, so is $\alpha^k(X\alpha\to Z \to Y\alpha)$, which shows the minimality statement. 
\end{proof}

We conclude this section with:

\begin{proposition}\label{char_cofib}
Let $\C$ be a model category.
Assume that each functor in $\tame(\pos,\C)$ can be discretized by a poset functor admitting a transfer.
The following statements about a natural transformation $\varphi\colon X\to Y$ in  $\tame(\pos,\C)$ are equivalent:
\begin{enumerate}
    \item $\varphi$ is a cofibration.
    \item For every finite full subposet $\alpha\colon\D\subset\pos$  admitting transfer and discretising both $X$and $Y$, the restriction $\varphi^\alpha\colon X\alpha\to Y\alpha $ is a cofibration
    in $\Fun(\D,\C)$.
\end{enumerate}
\end{proposition}

\begin{proof}
That statement~2 implies statement~1 is clear from the definition.
For the other implication, assume  $\varphi\colon X\to Y$ is a cofibration in $\tame(\pos,\C)$. Recall that cofibrations are all, and only, those morphisms with the left lifting property with respect to fibrations and weak equivalences (see~\cite{dwyerspalinski}).
Let $\alpha\colon \D\subset \C$ be a full subposet inclusion admitting transfer and discretizing both $X$ and $Y$. 
Consider the following commutative diagrams:
\[
\begin{tikzcd}
X\alpha\ar[swap,"\varphi^\alpha"]{d}\ar{r} & E\ar{d}{\psi}\\
Y\alpha\ar{r}\ar[dotted]{ur} & B
\end{tikzcd}\ \ \ \ \ \ \ \ \ \ \ \ 
\begin{tikzcd}
X\ar{d}[swap]{\varphi} & \alpha^k(X\alpha)\ar{l}[swap]{\cong}\ar{r}\ar{d}[swap]{\alpha^k(\varphi^\alpha)} & \alpha^k(E)\ar{d}{\alpha^k(\psi)}\\
Y & \alpha^k(Y\alpha)\ar{l}[swap]{\cong}\ar{r}\ar[dotted]{ur} & \alpha^k(B)
\end{tikzcd}
\]
where
\begin{itemize}
    \item $\psi$ is a fibration and a weak equivalence;
    \item the right square of the right diagram is the left Kan extension applied to the left diagram;
    \item the dotted arrow in the right diagram is a lift given by the assumption that $\varphi$ is a cofibration;
    \item the dotted arrow in the left diagram is the restriction of the dotted arrow in the right diagram along $\alpha$.
\end{itemize}
The existence of the dotted arrow in the left diagram shows that $\varphi^\alpha$ is a cofibration. 
\end{proof}

\section{Chain complexes of tame functors vs tame functors of chain complexes}\label{sec_chainVStame}
Consider a  functor  $X\colon \pos\to\ch(\C)$ in $\tame(\pos,\ch(\C))$. 
If $\D\subset  \pos$ is a finite subposet that discretizes  $X$, then, in every degree $n$, the subposet $\D$ also discretizes the functor $X_n\colon \pos\to\C$,  and, consequently, $X_n$  is  tame for every degree $n$.
For every $n$, the boundary morphisms  $\{X(x)_{n+1}\to X(x)_{n}\}_{x\in \pos}$ induce a natural transformation $X_{n+1}\to X_{n}$. 
These natural transformations form a chain complex and hence an object in $\ch(\tame(\pos, \C))$ which we denote by $I(X)$. 
In an analogous way, we can assign to a morphism in $\tame(\pos,\ch(\C))$ a morphism in  $ \ch(\tame(\pos, \C))$.
This defines a fully faithful and essentially injective functor which we denote by the symbol $I\colon \tame(\pos,\ch(\C))\subset \ch(\tame(\pos, \C))$.
Note that $I$ may fail to be essentially surjective. 
Indeed, one can observe that, by definition of $I$, the objects in $\ch(\tame(\pos, \C))$ that are in the image of $I$ are those chain complexes $X$ such that there exists a finite subposet $\D\subset\pos$ discretizing $X_n$ in every degree $n$. 
The existence of such a common $\D$ may fail if $X$ is unbounded.
Although $I$ is not essentially surjective, it still enjoys the following properties. 

\begin{proposition}\label{prop_Ipreservesstuff}
Assume   $\C$ is an abelian category satisfying the minimal cover axiom and $\pos$ is a poset such that every functor in 
$\tame(\pos,\ch(\C))$ can be discretized by a poset functor admitting a transfer. 
Then, with respect to the abelian and model structures on $\tame(\pos,\ch(\C))$ and $\ch(\tame(\pos, \C))$
given by Theorem~\ref{th_tame_ab} and Theorem~\ref{th_tame_mod}, the functor $I\colon \tame(\pos,\ch(\C))\to \ch(\tame(\pos, \C))$ has the following properties:
\begin{enumerate}
    \item A sequence in $\tame(\pos,\ch(\C))$ is exact if and only if applying $I$ to it returns an exact sequence in $\ch(\tame(\pos, \C))$.
    \item An object $X$ in $\tame(\pos,\ch(\C))$  is projective if and only if the object  $I(X)$ is projective in $\ch(\tame(\pos, \C))$. 
    \item A morphism $\phi$  is a weak equivalence/fibration/cofibration in   $\tame(\pos,\ch(\C))$
    if and only if  $I(\phi)$ is  a weak equivalence/fibration/cofibration in
    $\ch(\tame(\pos, \C))$.
\end{enumerate}
\end{proposition}
\begin{proof}
\textit{1}.\quad A sequence $0\to X\to Y\to Z\to 0$ is exact in $\tame(\pos,\ch(\C))$ if and only if $0\to X_n(x)\to Y_n(x)\to Z_n(x)\to 0$ is exact in $\C$ for every $x$ in $\pos$ and every natural number $n$. 
But this is also equivalent to exactness of $0\to I(X)\to I(Y)\to I(Z)\to 0$ in $\ch(\tame(\pos, \C))$. 

\noindent
\textit{2}.\quad  Since $I$ preserves epimorphisms (by statement 1) and is a full embedding, if $I(X)$ is projective in $\ch(\tame(\pos, \C))$, then $X$ is projective in $\tame(\pos, \ch(\C))$. 
For the reverse implication, consider a projective object $X$ in $\tame(\pos, \ch(\C))$, a morphism $I(X)\to Y$ and an epimorphism $Z\to Y$ in $\ch(\tame(\pos, \C))$. 
Let  $\alpha \colon \D \subset\pos$ be a finite full subposet inclusion that discretizes $X$. Then  the natural transformation  $\alpha^k (X\alpha)\to X$ is an isomorphism (see Proposition~\ref{prop_full_lke}) and consequently so is
$\alpha^k (I(X)\alpha)\to I(X)$, where $\alpha^k(I(X)\alpha)$ denotes the chain complex obtained by applying first precomposition by $\alpha$ and then $\alpha^k$ degree-wise to $X$ (in degree $n$ it is given by $\alpha^k(X_n\alpha)$).
 
Both left Kan extensions and precomposition preserve epimorphisms, hence the epimorphism $Z\to Y$ in $\ch(\tame(\pos, \C))$ is sent to an epimorphism $\alpha
^k(Z\alpha)\to \alpha^k(Y\alpha)$ in $\tame(\pos, \ch(\C))$. 
This gives rise to the following diagram: 
\[
\begin{tikzcd}
& & \alpha^k(Z\alpha)\ar[dd]\ar[dr] & \\
& & & Z\ar[dd]\\
\alpha^k(I(X)\alpha)\ar[rr]\ar[dashed, uurr, bend left=25pt]\ar[swap,"\cong", dr] & & \alpha^k(Y\alpha)\ar[dr] & \\
& I(X)\ar[rr]\ar[dashed, uurr, crossing over, bend left=25pt] & & Y
\end{tikzcd}
\]
where the top dashed arrow is constructed using the projectiveness of 
$X$ in $\tame(\pos, \ch(\C))$, and the other dashed arrow,  obtained by composition, is a desired lift in $\ch(\tame(\pos, \C))$. 

\noindent
\textit{3}.\quad The functor $I$ preserves fibrations and weak equivalences because they are defined pointwise.
Since $I$ in addition is a full inclusion, 
a cofibration in $\ch(\tame(\pos, \C)$ is also a cofibration in $\tame(\pos, \ch(\C))$. 
Assume that $\varphi\colon X\to Y$ is a cofibration in $\tame(\pos, \ch(\C))$. 
To show that $I(\varphi)\colon I(X)\to I(Y)$ is a cofibration in $\ch(\tame(\pos, \C))$, we take $\alpha $ discretizing both $X$ and $Y$ (it exists by definition of cofibration, see Theorem~\ref{th_tame_mod}) so to have $X\cong \alpha^k(X\alpha)$ and $Y\cong \alpha^k(Y\alpha)$, and consider the following diagram
\[
\begin{tikzcd}
I(X)\ar["\cong", r]\ar[bend left=20pt, rrr]\ar["I(\varphi)"', d] 
& \alpha^k(I(X)\alpha)\ar[r]\ar["\alpha^k(\varphi^\alpha)"',d]
& \alpha^k(E\alpha) \ar[r]\ar["\alpha^k(\psi^\alpha)",d] 
& E\ar["\psi",d]
\\
I(Y)\ar["\cong", r]\ar[bend right=15pt, rrr] 
& \alpha^k(I(Y)\alpha)\ar[r]\ar[dashed, ur]
&  \alpha^k(B\alpha) \ar[r] 
& B
\end{tikzcd}
\]
Here, $\psi$ is a fibration and weak equivalence. So, to see that $I(\varphi)$ is a cofibration, we just need to show that a lift $I(Y)\to E$ exists making the outer square commute. 
By definition, $\alpha^k(\varphi^\alpha)$ is a cofibration in $\tame(\pos, \ch(\C))$ because $\varphi$ is so, and $\alpha^k(\psi^\alpha)$ is a fibration and weak equivalence because $\psi$ is so, and both precomposition and left Kan extension preserve fibrations and weak equivalences. 
Thus, the central diagram in $\tame(\pos, \ch(\C))$ admits the lift represented by the dashed arrow. 
A desired lift in $\ch(\tame(\pos, \C))$ is then obtained by composition. 
\end{proof}

\section{Chain complexes, tameness, and dimension 1}\label{subs_unify}
In this section, we summarize what was presented in previous sections under additional assumptions.
Namely, we assume that $\pos$ is a finite type poset of dimension at most $1$ such that, for every finite $\D\subset \pos$, the set $ \text{\rm suplim}(\D)$ is also finite. 
We also assume that $\C$ is an abelian category of projective dimension 0, for example, the category of vector spaces over a field  $\mathbb F$. 
Since all objects in $\C$ are projective, the category $\C$ satisfies the minimal cover axiom (see Section~\ref{subs_abmin}).  
Using $\pos$ and $\C$, under these assumptions, we can form the following categories with various structures:

\noindent\rule{\textwidth}{0.5pt}

\vspace{1mm}
\noindent
\fbox{$\R(\pos)$:} \quad the realization of the poset $\pos$ (see Definition~\ref{sdgdfhdfgh}), which has dimension at most~1 (see Proposition~\ref{prop_real_dim1})
\\
\noindent\rule{\textwidth}{0.5pt}

\vspace{1mm}
\noindent
\fbox{$\tame(\R(\pos), \C)$:}\quad the category of tame functors $X\colon \R(\pos)\to\C$. It  enjoys the following properties:
    \begin{itemize}
        \item All its objects can be discretized by a poset functor admitting a transfer (see Theorem~\ref{adfgsdfjdghn}).
        \item It is an abelian category satisfying the minimal cover axiom, and it has projective dimension at most 1 (see Theorem~\ref{th_tame_ab}).
        \item An object in $\tame(\R(\pos), \C)$ is projective if and only if it is isomorphic to a functor of the form
        $\bigoplus_{x\in\D} P_x(x,-)$ for some sequence  $\{P_x\}_{x\in\D}$ of objects in $\C$ indexed by a finite subset  $\D\subset \R(\pos)$
        (see Corollary~\ref{asfddhdhfj}). \vspace{-3mm}
    \end{itemize}
\noindent\rule{\textwidth}{0.5pt}

\vspace{1mm}
\noindent    
\fbox{$\ch(\C)$:}\quad the category of non-negative chain complexes in the abelian category $\C$. It enjoys the following properties:
\begin{itemize}
    \item It is an abelian category satisfying the minimal cover axiom (see Proposition~\ref{prop:minimal_cover}).
    \item An object $X$ in $\ch(\C)$ is projective if and only if $H_n(X)=0$ for all $n>0$, or equivalently if and only if $X$ is isomorphic to a functor of the form $\bigoplus _{i\in I}D(A_i)$ for some sequence of objects $\{A_i\}_{i\in I}$ in $\C$
    (see Proposition~\ref{prop:projective_characterization}).
     \item It is a model category (see Proposition~\ref{prop:ch_model}) satisfying the minimal cofibrant factorization axiom (see Corollary~\ref{cor:ch-min-cofbnt-replmnt}) with all objects being cofibrant.
     \vspace{-3mm}
\end{itemize}
\noindent\rule{\textwidth}{0.5pt}

\vspace{1mm}
\noindent    
\fbox{$\ch(\tame(\R(\pos), \C))$:}\quad the category of non-negative chain complexes in the abelian category $\tame(\R(\pos), \C)$. 
It enjoys the following properties:
\begin{itemize}
    \item It is an abelian category satisfying the minimal cover axiom (see Proposition~\ref{prop:minimal_cover} and Theorem~\ref{th_tame_ab}).
    \item An object $X$ in $\ch(\tame(\R(\pos), \C))$ is projective if and only if 
    $H_n(X)=0$ for all $n>0$, and $H_0(X)$ and $X_n$ are projective in $\tame(\R(\pos), \C)$
    for all $n\geq 0$, 
    or equivalently if and only if $X$ is isomorphic to a functor of the form 
    $\bigoplus _{i\in I}D(A_i)$ where $\{A_i\}_{i\in I}$ is a sequence of projective objects in $\tame(\R(\pos), \C)$
    (see Proposition~\ref{prop:projective_characterization});
    \item It is a model category (see Proposition~\ref{prop:ch_model}) satisfying the minimal cofibrant factorization axiom (see Corollary~\ref{cor:ch-min-cofbnt-replmnt}).
\item Every cofibrant object $X$ in $\ch(\tame(\R(\pos), \C))$ is isomorphic to $\bigoplus_{n\geq 0}S^n(P[n]) \oplus 
D^{n+1}(Y_n)$, where $P[n]$ is a minimal resolution of $H_n(X)$ in $\tame(\R(\pos), \C)$ (necessarily of length at most 1) and $Y_n$ is a projective object in $\tame(\R(\pos), \C)$  for every $n\geq 0$ (see Theorem~\ref{adgfgjhjk} and recall that, by Theorem~\ref{th_tame_ab}.4, the category $\tame(\R(\pos), \C)$ is  abelian  of projective dimension at most 1).
\vspace{-3mm}
 \end{itemize}
\noindent\rule{\textwidth}{0.5pt}

\vspace{1mm}
\noindent    
\fbox{$\tame(\R(\pos),\ch(\C))$:}\quad  the category of tame functors $X\colon \R(\pos)\to\ch(\C)$. It enjoys the following properties:
\begin{itemize}
    \item All its objects can be discretized by a poset functor admitting a transfer (see Theorem~\ref{adfgsdfjdghn});
    \item It is an abelian category satisfying the minimal cover axiom(see Theorem~\ref{th_tame_ab} and Proposition~\ref{prop:minimal_cover}).
    \item An object in $\tame(\R(\pos),\ch(\C))$ is projective if and only if it is isomorphic to a functor of the form
    $\bigoplus_{x\in\D} P_x(x,-)$ for some sequence  $\{P_x\}_{x\in\D}$ of projective objects in $\ch(\C)$ indexed by a finite subset  $\D\subset \R(\pos)$ (see Corollary~\ref{asfddhdhfj}).
    \item It is a model category satisfying the minimal cofibrant factorization axiom (see Theorem~\ref{cor:ch-min-cofbnt-replmnt} and Theorem~\ref{th_tame_mod}).
    \item Every cofibrant object in $\tame(\pos, \ch(\C))$ is isomorphic to a direct sum of functors whose values are chain complexes non-zero in at most two consecutive degrees with differentials being monomorphisms (use the structure Theorem~\ref{adgfgjhjk} for cofibrant objects in $\ch(\tame(\pos, \C))$ and Proposition~\ref{prop_Ipreservesstuff} to see that an object in $\tame(\pos, \ch(\C))$ is cofibrant if and only if $I(X)$ is cofibrant in $\ch(\tame(\pos, \C))$).
\end{itemize}
\noindent\rule{\textwidth}{0.5pt}
\vspace{1mm}

We summarize this discussion as follows:

\begin{theorem}\label{thm_indecomposables}
Assume: 
\begin{itemize}
    \item $\C$ is an abelian category where every object is projective;
    \item $\pos$ is a finite type poset of dimension at most $1$ such that, for every finite $\D\subset \pos$, the set $ \text{\rm suplim}(\D)$ is also finite. 
\end{itemize}
Then every cofibrant object in $\tame(\R(\pos),\ch(\C))$ is isomorphic to a direct sum of indecomposable cofibrant objects. Moreover, an object $X$ in $\tame(\R(\pos),\ch(\C))$ is indecomposable cofibrant if and only if there is a natural number $n$ such that 
one of the following conditions holds:
\begin{enumerate}
    \item $H_\ast X=0$, in which case  $X$ is isomorphic to $D^{n+1}(A)$ 
     for $A$ an indecomposable projective object in $\tame(\R(\pos), \C)$;
    \item $H_nX\neq 0$, in which case $H_nX$ is indecomposable in $\tame(\R(\pos), \C)$ and 
    $X$ is isomorphic to $S^n(P)$ where $P$
    is a minimal projective resolution of $H_nX$ in $\tame(\R(\pos), \C)$. 
\end{enumerate}
\end{theorem}

\section{Generating indecomposables by gluing}\label{sec_gluing}
In this section, we exhibit a general method to construct indecomposables in a functor category.
We give conditions under which it is possible to extend the indexing poset of an indecomposable by gluing another object to it while preserving indecomposability. 
The results of this section hold more in general than the ones in the previous sections: no condition is imposed on the indexing poset, and the landing category $\C$ is only assumed to be abelian.
\medskip

Recall that the \define{radical} of a  functor $X\colon \D\to \C$, indexed by  a finite poset $\D$,
can be defined as the  subfunctor $\rad X\subset X$ 
that at each $x$ in $\D$ is equal to $\text{im}\left(\bigoplus_{y<x} X(y)\to X(x)\right)$~\cite[Ch. III, Lemma 2.2]{assem2006elements}.
The following is a well-known result characterizing indecomposable functors (see, for example, \cite[Section 3.4]{jacobson}).

\begin{lemma}\label{lem:indecomposable}
Let $\C$ be an abelian category and $\D$ be a finite poset.
The following statements are equivalent for a functor $X\colon \D\to\C$:
\begin{itemize}
    \item $X$  is indecomposable.
    \item The monoid $\End (X)$ of endomorphisms on $X$  does not contain any non-trivial idempotent elements.
    \item Every endomorphism of $X$ is either an isomorphism or factors through the subfunctor $\rad X\subset  X$.
    \item Every endomorphism of $X$ is either an isomorphism or is nilpotent.
\end{itemize}
\end{lemma}

For a functor $X\colon \D\to \C$ and a subposet $\A\subseteq \D$, we denote by
$X_\A$ the restriction of $X$ to $\A$. 

\begin{theorem}\label{th_gluing}
Let $\D$ be a finite poset and $\A,\B\subset \D$ be full subposets such that  as set $\D =\A\cup \B$. 
Consider   a functor
 $X\colon \D \to \C$   and the unique morphism 
 $\beta\colon i^k X_{\A\cap \B}\to X_\B$  given by the universality of the left Kan extension along the inclusion $i\colon \A\cap \B\to \B$.
Assume  $X_\A$ is indecomposable. Then
the following statements are equivalent:
\begin{enumerate}
    \item $X$ is indecomposable.
    \item The morphism $\rad X_\B\subset  X_\B$ induces an isomorphism between $\hom (\coker \beta, \rad X_\B)$ and $\hom (\coker \beta, X_\B)$.
    \item The kernel of the restriction homomorphism $\End (X_\B)\to \End(X_{\A\cap\B})$ consists of nilpotent elements.
\end{enumerate}
\end{theorem}

\begin{proof}
{\em 1}$\Rightarrow${\em 2}:\quad  
Let $X$ be indecomposable. 
For a natural transformation $\phi\colon \coker \beta\to X_\B$, consider the endomorphism  $\overline{\phi}\colon X_\B\to X_\B$ given by the composition of the quotient $X_\B\to \coker \beta$ and $\phi$.
Note that the restriction of $\overline{\phi}$ to $\A\cap \B$ is the $0$ natural transformation. Thus, there is $\psi\colon X\to X$ whose restrictions to $\A$ and $\B$ are, respectively, the 
$0$ natural transformation and $\overline{\phi}$. 
Since $X$ is indecomposable, by Lemma~\ref{lem:indecomposable}, $\psi$ is either an isomorphism or factors through $\rad X\subset X$.
The first case is not possible since $X_\A$ is indecomposable, hence non-zero and yet  $\psi$ restricted to $\A$ is the $0$ natural transformation. The natural transformation $\psi$ factors therefore through $\rad X\subset X$ and consequently
so does its restriction $\overline{\phi}$ to $\B$. This means that 
$\phi\colon \coker \beta\to X_\B$ also factors through $\rad X_\B\subset X_\B$ which is what we needed to show.
\smallskip

\noindent
{\em 2}$\Rightarrow${\em 3}:\quad  
Let $\psi\colon X_\B\to X_\B$ be an endomorphism whose restriction to $\A\cap\B$ is the $0$ natural transformation. We can therefore form the following commutative diagram:
\[
\begin{tikzcd}[ampersand replacement=\&]
i^k X_{\A\cap \B} \arrow[r, "\beta"] \arrow[d, "0"'] 
\& X_\B \arrow[d, "\psi"] 
\\
i^k X_{\A\cap \B}  \arrow[r, "\beta"] 
\& X_\B
\end{tikzcd}\]
The morphisms $\psi$ factors therefore 
through a morphism of the form $\coker \beta\to X_\B$, which by assumption factors through $\rad X_\B\subset X_\B$. Consequently, by Lemma~\ref{lem:indecomposable} and since $\psi$ is not an isomorphism, $\psi$ is nilpotent.
\smallskip

\noindent
{\em 3}$\Rightarrow${\em 1}:\quad 
Consider an idempotent $\psi\colon X\to X$. According to Lemma~\ref{lem:indecomposable}, we need to prove that either $\psi=0$ or $\psi=\text{id}$. 
Since $X_\A$ is indecomposable, there are two cases: either $\psi_\A=0$ or $\psi_\A=\text{id}$. In the first case, $\psi_\B$
is both idempotent (since $\psi$ is idempotent) and nilpotent (by assumption 3)
and hence it has to be $0$, and thus so is $\psi$.
In the second case, $\text{id}-\psi_\B$
is both idempotent and nilpotent, and as before
$\text{id}-\psi_\B$  has to be $0$, and thus 
$\psi=\text{id}$.
\end{proof}

As a direct consequence, we obtain: 

\begin{corollary}\label{asdvadfvdfv}
Under the assumptions of Theorem~\ref{th_gluing}:
\begin{enumerate}
\item If $\hom (\coker \beta, X_\B)=0$, then
$X$ is indecomposable.
\item If the restriction homomorphism $\End(X_\B)\to \End(X_{\A\cap \B})$ is an inclusion, then $X$ is indecomposable.
\end{enumerate}
\end{corollary}

It turns out that the assumptions of both of the statements in Corollary~\ref{asdvadfvdfv} are equivalent. 
To show this we need a lemma.

\begin{lemma}\label{hom_coker_zero}
There exists at most one $\nu$ making the following diagram commute if and only if 
$\hom\left(\coker\beta, X\right) =0$:
\begin{equation}\label{D_unique_extention}
\begin{tikzcd}[ampersand replacement=\&]
Y \arrow[r, "\beta"] \arrow[d, "\eta"'] 
\& X \arrow[d, dotted, "\nu"] 
\\
Y \arrow[r, "\beta"] 
\& X
\end{tikzcd}  
\, .
\end{equation}
\end{lemma}

\begin{proof}
Consider the following right exact sequence: $Y\xrightarrow{\beta}X\twoheadrightarrow \coker\beta$. 
By applying $\hom(-,X)$ and $\hom(-,Y)$ to it, we obtain the following commutative diagram:
\begin{equation*}  
\begin{tikzcd}[ampersand replacement=\&]
\hom(Y,X) 
\& \hom(X,X) \ar[l, "\beta^\ast"'] 
\& \hom(\coker\beta,X)\ar[l]
\& 0 \ar[l]
\\
\hom(Y,Y) \ar[u]
\& \hom(X,Y) \ar[l] \ar[u]
\& \hom(\coker\beta, Y)\ar[u] \ar[l]
\& 0 \ar[l]
\end{tikzcd}\, .
\end{equation*}
The condition $\hom(\coker\beta,X)=0$ is equivalent to $\beta^\ast$ being a monomorphism. 
Let us take $\xi\in\hom(Y,X)$ to be the image of $\eta \in \hom(Y,Y)$ along the leftmost vertical map. 
The preimage of $\xi$ under $\beta^*$ is either empty or exactly one morphism $\nu$. 
\end{proof}

\begin{proposition}\label{sdgsdghfg}
Let $\B$ be a finite poset and $\A\subset \B$ be its subposet.
Consider a functor $X\colon \B \to \C$ and the unique morphism $\beta\colon i^k X_{\A}\to X$ given by the universality of the left Kan extension along the inclusion $i\colon \A\subset  \B$. Then the following statements are equivalent:
\begin{enumerate}
    \item $\hom (\coker \beta, X_\B)=0$.
    \item The restriction homomorphism $\End(X)\to \End(X_{\A})$ is an inclusion.
\end{enumerate}
\end{proposition}

Next, we exemplify how to construct complex indecomposables using this technique. 
This example also shows that Theorem~\ref{adgfgjhjk} does not hold for projective dimension strictly greater than $1$.

\begin{figure}[ht]
\centering
\adjustbox{scale=.8}{
\begin{tikzpicture}[on top/.style={preaction={draw=white,-,line width=#1}},
on top/.default=4pt]
\node (a00)[] at (11,8.8) {$\mathbb{F}$};
\node (a0)[] at (11,7.5) {$\mathbb{F}^2$};
\node (a1)[] at (12,7) {$\mathbb{F}$};
\node (a2)[] at (9.5,6.5) {$\mathbb{F}$};
\node (a3)[] at (11,6) {$\mathbb{F}$};
\node (a4)[] at (12,5.5) {$\mathbb{F}$};
\node (a5)[] at (9.5,5) {$\mathbb{F}$};
\node (a6)[] at (10.5,4.5) {$\mathbb{F}$};
\node (a7)[] at (4.5,5.5) {$\mathbb{F}$};
\node (a8)[] at (2.5,5) {$\mathbb{F}$};
\node (a9)[] at (6,5) {$\mathbb{F}^3$};
\node (a10)[] at (4.5,4) {$\mathbb{F}^3$};
\node (a11)[] at (6,3.7) {$\mathbb{F}$};
\node (a12)[] at (8,4) {$\mathbb{F}$};
\node (a13)[] at (9,3.5) {$\mathbb{F}$};
\node (a14)[] at (2.5,3.7) {$\mathbb{F}$};
\node (a15)[] at (2.8,3) {$\mathbb{F}$};
\node (a16)[] at (.85,2.7) {$\mathbb{F}$};
\node (a17)[] at (4.5,2.7) {$\mathbb{F}$};
\node (a18)[] at (6.5,3) {$\mathbb{F}$};
\node (a19)[] at (7.5,2.5) {$\mathbb{F}$};
\node (a20)[] at (2.8,1.6) {$\mathbb{F}$};
\node (a21)[] at (6,6.5) {$\mathbb{F}$};
\draw[->]  (a00) -- node [pos=.4,label={[label distance=-.2cm]left:\scalemath{0.6}{\begin{bmatrix} 0 \\ 1 \end{bmatrix}}}] {} (a0);
\draw[->]  (a0) -- node [pos=.6,label={[label distance=-.2cm]right:\scalemath{0.6}{\begin{bmatrix} 1 & 0 \end{bmatrix}}}] {} (a3);
\draw[->]  (a2) -- node [pos=0.3,label=above:\scalemath{0.6}{\begin{bmatrix} 1 \\ 0 \end{bmatrix}}] {} (a0);
\draw[->]  (a1) -- node [pos=0.3,label=above:\scalemath{0.6}{\begin{bmatrix} 1 \\ 1 \end{bmatrix}}] {} (a0);
\draw[->]  (a5) -- (a3);
\draw[->]  (a4) -- (a3);
\draw[->]  (a2) -- (a5);
\draw[->]  (a1) -- (a4);
\draw[->]  (a6) -- (a5);
\draw[->]  (a6) -- (a4);
\draw[->]  (a13) -- (a12);
\draw[->]  (a13) -- (a6);
\draw[->]  (a12) -- (a5);
\draw[->]  (a19) -- (a13);
\draw[->]  (a19) -- (a18);
\draw[->]  (a18) -- (a12);
\draw[->]  (a12) -- node [pos=.4,label={[label distance=-.1cm]above:\scalemath{0.6}{\begin{bmatrix} 0 \\ 0 \\ 1 \end{bmatrix}}}] {} (a9);
\draw[->]  (a10) -- (a9);
\draw[->]  (a10) -- node [pos=.4,label={[label distance=-.15cm]right:$f$}] {} (a17);
\draw[->]  (a8) -- node [pos=.5,label={[label distance=-.1cm]above:\scalemath{0.6}{\begin{bmatrix} 1 \\ 0 \\ 0 \end{bmatrix}}}] {} (a10);
\draw[->]  (a8) -- (a14);
\draw[->]  (a7) -- (a21);
\draw[->]  (a7) -- node [pos=.4,label={[label distance=-.2cm]right:\scalemath{0.6}{\begin{bmatrix} 0 \\ 1 \\ 0 \end{bmatrix}}}] {} (a10);
\draw[->]  (a21) -- node [pos=.4,label={[label distance=-.2cm]right:\scalemath{0.6}{\begin{bmatrix} 0 \\ 1 \\ 0 \end{bmatrix}}}] {} (a9);
\draw[->]  (a16) -- (a14);
\draw[->]  (a16) -- (a20);
\draw[->]  (a20) -- (a17);
\draw[->]  (a15) -- (a20);
\draw[->]  (a14) -- (a17);
\draw[->]  (a17) -- (a11);
\draw[->]  (a9) -- node [pos=.4,label={[label distance=-.15cm]right:$f$}] {} (a11);
\draw[->, on top]  (a15) -- node [pos=.3,label={[label distance=-.1cm]above:\scalemath{0.6}{\begin{bmatrix} 1 \\ 1 \\ 1 \end{bmatrix}}}] {} (a10);
\draw[->, on top]  (a18) -- node [pos=0.3,label=below:\scalemath{0.6}{\begin{bmatrix} 0 \\ 0 \\ 1 \end{bmatrix}}] {} (a10);
\end{tikzpicture}
}
\caption{A parametrized chain complex $X$ (indexed by the poset in Figure~\ref{subfig_poset_ind_3}). The displayed maps are identities if not otherwise specified, and $f=[1 \ 0 \ 0]$. The vertical maps are chain maps. Everything not displayed is zero.}
\label{diag_the_counterexample}
\end{figure}

\begin{corollary}\label{monster_is_indecomposable}
    The object $X$ in $\Fun(\D, \ch(\vect))$, defined as shown in Figure~\ref{diag_the_counterexample}, is indecomposable.
\end{corollary}

\begin{proof}
The proof follows by repeatedly applying Corollary~\ref{asdvadfvdfv} to the restrictions of $X$ to the series of posets depicted in Figure~\ref{fig:gluing-posets} since $X_\A$ in step (a) is clearly indecomposable.
Indeed, in step (a) $i^k X_{\A\cap \B}$ is nonzero 
only in degree $0$,  and in step (c) it is nonzero only in degree $1$. 
Therefore, $\coker\beta$ is zero in the same degrees. 
So, by commutativity, the only possible morphism between $\coker\beta$ and $X_\B$ is the zero morphism. 
In step (b), we see that $\hom(\coker\beta, X_{\mathcal{B}})=0$ by direct verification.
\end{proof} 

\begin{figure}[ht]
\begin{subfigure}[b]{.18\linewidth}
\centering
\adjustbox{scale=0.7}{
\begin{tikzpicture}
    \node[draw,shape=circle,fill=black,scale=0.3] (x1) {};
    \node[draw,shape=circle,fill=black, above left=1cm of x1,scale=0.3] (x2) {};
    \node[draw,shape=circle,fill=black, above right=1cm of x1,scale=0.3] (x3) {};
    \node[draw,shape=circle,fill=black, above left=1cm of x3,scale=0.3] (x4) {};
    \node[below=0.55cm of x1,scale=0.3] (xy) {};
    \node[draw=red, thick,fit=(x2),inner sep=1.3ex,circle,label=left:{\textcolor{red}{$\A$}}] (A) {};
    \node[above=0.4cm of x3] (B) {\textcolor{blue}{$\mathcal{B}$}};
    \draw[draw=blue, rounded corners, thick]([shift={(-6mm,+3mm)}]x1.west) -- ([shift={(-3mm,+0mm)}]x2.west) -- ([shift={(+0mm,+3mm)}]x4.north) -- ([shift={(+3mm,+0mm)}]x3.east) -- ([shift={(+0mm,-3mm)}]x1.south) -- ([shift={(-6mm,+3mm)}]x1.west);
    \draw[->, shorten >=2pt, shorten <=2pt] (x2) -- (x1);
    \draw[->, shorten <=2pt, shorten >=2pt] (x2) -- (x4);
    \draw[->, shorten >=2pt, shorten <=2pt] (x1) -- (x3);
    \draw[->, shorten >=2pt, shorten <=2pt] (x4) -- (x3);
\end{tikzpicture}
}
\caption{}
\label{subfig_poset_ind_1}
\end{subfigure}
\begin{subfigure}[b]{.35\linewidth}
\centering
\adjustbox{scale=0.7}{
\begin{tikzpicture}
    \node[draw,shape=circle,fill=black,scale=0.3] (x1) {};
    \node[draw,shape=circle,fill=black, above left=1cm of x1,scale=0.3] (x2) {};
    \node[draw,shape=circle,fill=black, above right=1cm of x1,scale=0.3] (x3) {};
    \node[draw,shape=circle,fill=black, above left=1cm of x3,scale=0.3] (x4) {};
    \node[draw,shape=circle,fill=black, above right=1cm of x3,scale=0.3] (x5) {};
    \node[draw,shape=circle,fill=black, below right=1cm of x3,scale=0.3] (x6) {};
    \node[draw,shape=circle,fill=black, above right=1cm of x6,scale=0.3] (x7) {};
    \node[draw,shape=circle,fill=black, below right=1cm of x6,scale=0.3] (x8) {};
    \node[draw,shape=circle,fill=black, above right=1cm of x8,scale=0.3] (x9) {};
    \node[above=0.4cm of x2] (A) {\textcolor{red}{$\A$}};
    \draw[draw=red, rounded corners, thick]([shift={(-6mm,+3mm)}]x1.west) -- ([shift={(-3mm,+0mm)}]x2.west) -- ([shift={(+0mm,+3mm)}]x4.north) -- ([shift={(+3mm,+0mm)}]x3.east) -- ([shift={(+0mm,-3mm)}]x1.south) -- ([shift={(-6mm,+3mm)}]x1.west);
    \node[above right=0.2cm of x7] (B) {\textcolor{blue}{$\mathcal{B}$}};
    \draw[draw=blue, rounded corners, thick]([shift={(-6mm,+3mm)}]x6.west) -- ([shift={(-3mm,+0mm)}]x3.west) -- ([shift={(+0mm,+3mm)}]x5.north) -- ([shift={(+3mm,+0mm)}]x9.east) -- ([shift={(+0mm,-3mm)}]x8.south) -- ([shift={(-6mm,+3mm)}]x6.west);
    \draw[->, shorten >=2pt, shorten <=2pt] (x2) -- (x1);
    \draw[->, shorten <=2pt, shorten >=2pt] (x2) -- (x4);
    \draw[->, shorten >=2pt, shorten <=2pt] (x1) -- (x3);
    \draw[->, shorten >=2pt, shorten <=2pt] (x4) -- (x3);
    \draw[->, shorten >=2pt, shorten <=2pt] (x6) -- (x3);
    \draw[->, shorten >=2pt, shorten <=2pt] (x3) -- (x5);
    \draw[->, shorten <=2pt, shorten >=2pt] (x6) -- (x7);
    \draw[->, shorten >=2pt, shorten <=2pt] (x7) -- (x5);
    \draw[->, shorten >=2pt, shorten <=2pt] (x8) -- (x6);
    \draw[->, shorten >=2pt, shorten <=2pt] (x8) -- (x9);
    \draw[->, shorten >=2pt, shorten <=2pt] (x9) -- (x7);    
\end{tikzpicture}
}
\caption{}
\label{subfig_poset_ind_2}
\end{subfigure}
\begin{subfigure}[b]{.45\linewidth}
\centering
\adjustbox{scale=0.7}{
\begin{tikzpicture}
    \node[draw,shape=circle,fill=black,scale=0.3] (x1) {};
    \node[draw,shape=circle,fill=black, above left=1cm of x1,scale=0.3] (x2) {};
    \node[draw,shape=circle,fill=black, above right=1cm of x1,scale=0.3] (x3) {};
    \node[draw,shape=circle,fill=black, above left=1cm of x3,scale=0.3] (x4) {};
    \node[draw,shape=circle,fill=black, above right=1cm of x3,scale=0.3] (x5) {};
    \node[draw,shape=circle,fill=black, below right=1cm of x3,scale=0.3] (x6) {};
    \node[draw,shape=circle,fill=black, above right=1cm of x6,scale=0.3] (x7) {};
    \node[draw,shape=circle,fill=black, below right=1cm of x6,scale=0.3] (x8) {};
    \node[draw,shape=circle,fill=black, above right=1cm of x8,scale=0.3] (x9) {};
    \node[draw,shape=circle,fill=black, above right=1cm of x7,scale=0.3] (x10) {};
    \node[draw,shape=circle,fill=black, above right=1cm of x9,scale=0.3] (x11) {};
    \node[draw,shape=circle,fill=black, above right=1cm of x10,scale=0.3] (x12) {};
    \node[draw,shape=circle,fill=black, above right=1cm of x11,scale=0.3] (x13) {};
    \node[above=0.4cm of x3] (A) {\textcolor{red}{$\A$}};
    \draw[draw=red, rounded corners, thick] ([shift={(-6mm,+3mm)}]x6.west) -- ([shift={(0mm,-3mm)}]x3.south) -- ([shift={(0mm,-3mm)}]x1.south) -- ([shift={(-3mm,0mm)}]x2.west) -- ([shift={(0mm,+3mm)}]x4.north) -- ([shift={(+0mm,+3mm)}]x3.north) -- ([shift={(+0mm,+3mm)}]x5.north) -- ([shift={(+3mm,+0mm)}]x9.east) -- ([shift={(+0mm,-3mm)}]x8.south) -- ([shift={(-6mm,+3mm)}]x6.west);
    \node[below right=0.2cm of x11] (B) {\textcolor{blue}{$\mathcal{B}$}};
    \draw[draw=blue, rounded corners, thick]([shift={(-7mm,+3.5mm)}]x8.west) -- ([shift={(-3.5mm,+0mm)}]x6.west) -- ([shift={(+0mm,+3.5mm)}]x12.north) -- ([shift={(+3.5mm,+0mm)}]x13.east) -- ([shift={(+0mm,-3.5mm)}]x8.south) -- ([shift={(-7mm,+3.5mm)}]x8.west);
    \draw[->, shorten >=2pt, shorten <=2pt] (x2) -- (x1);
    \draw[->, shorten <=2pt, shorten >=2pt] (x2) -- (x4);
    \draw[->, shorten >=2pt, shorten <=2pt] (x1) -- (x3);
    \draw[->, shorten >=2pt, shorten <=2pt] (x4) -- (x3);
    \draw[->, shorten >=2pt, shorten <=2pt] (x6) -- (x3);
    \draw[->, shorten >=2pt, shorten <=2pt] (x3) -- (x5);
    \draw[->, shorten <=2pt, shorten >=2pt] (x6) -- (x7);
    \draw[->, shorten >=2pt, shorten <=2pt] (x7) -- (x5);
    \draw[->, shorten >=2pt, shorten <=2pt] (x8) -- (x6);
    \draw[->, shorten >=2pt, shorten <=2pt] (x8) -- (x9);
    \draw[->, shorten >=2pt, shorten <=2pt] (x9) -- (x7);    
    \draw[->, shorten >=2pt, shorten <=2pt] (x7) -- (x10);
    \draw[->, shorten <=2pt, shorten >=2pt] (x9) -- (x11);
    \draw[->, shorten >=2pt, shorten <=2pt] (x11) -- (x10);
    \draw[->, shorten >=2pt, shorten <=2pt] (x10) -- (x12);
    \draw[->, shorten >=2pt, shorten <=2pt] (x11) -- (x13);
    \draw[->, shorten >=2pt, shorten <=2pt] (x13) -- (x12);    
\end{tikzpicture}
}
\caption{}
\label{subfig_poset_ind_3}
\end{subfigure}
\caption{Posets used to construct the indecomposable object shown in Figure~\ref{diag_the_counterexample}.}
\label{fig:gluing-posets}
\end{figure}

\paragraph{\textbf{Acknowledgment}}
Wojciech Chach\'olski and Francesca Tombari were partially supported by VR, the Wallenberg AI, Autonomous System and Software Program (WASP) funded by Knut and Alice Wallenberg Foundation, and MultipleMS funded by the European Union under the Horizon 2020 program, grant agreement 733161, and dBRAIN collaborative project at digital futures at KTH.
Barbara Giunti was supported by the Austrian Science Fund (FWF)  P 29984-N35 and P 33765-N. 
Claudia Landi is a member of GNSAGA-INdAM and partially carried out this work within the activities of ARCES (University of Bologna).
Barbara Giunti thanks Luis Scoccola and Lukas Waas for discussions about Section~\ref{sec_gluing} and comparison with finite enconding, respectively.

\paragraph{Compliance with Ethical Standards.}
The authors have no conflicts of interest to disclose and all gave informed consent for the publication of the manuscript as submitted.  
This research did not involve Human Participants and/or Animals.

\printbibliography

\end{document}